\documentclass[article]{aiaa-pretty}    

\usepackage{amssymb}
\usepackage{amsmath}
\usepackage{booktabs}
 \usepackage{varioref}
 \usepackage{wrapfig}
 \usepackage{threeparttable}
 \usepackage{dcolumn}
  \newcolumntype{d}{D{.}{.}{-1}}
 \usepackage{nomencl}
  \makeglossary
 \usepackage{subfigure}
 \usepackage{graphicx}
 \usepackage{subfigmat}
 \usepackage{fancyvrb}
  \fvset{fontsize=\footnotesize,xleftmargin=2em}
 \usepackage{lettrine}
\usepackage{cite}
\usepackage{enumitem}

  \usepackage{letltxmacro}
\LetLtxMacro{\originaleqref}{\eqref}
\renewcommand{\eqref}{Eq.~\originaleqref}

\newtheorem{theorem}{Theorem}

\newtheorem{remark}{Remark}
\newtheorem{lemma}{Lemma}
\newtheorem{definition}{Definition}

\newenvironment{proof}{{\it Proof. }}{\hfill $\Box$}
\newtheorem{corollary}{Corollary}
\newtheorem{example}{Example}

\newcommand{\Real}{\mathbb R}
\newcommand{\set}[1]{\left\{#1\right\}}
\newcommand{\real}[1]{{\mathbb R}^{#1}}
\newcommand{\bd}{{\boldsymbol d}}
\newcommand{\be}{{\boldsymbol e}}
\newcommand{\bff}{{\boldsymbol f}}
\newcommand{\bg}{{\boldsymbol g}}

\newcommand{\bk}{{\boldsymbol k}}

\newcommand{\bp}{{\boldsymbol p}}

\newcommand{\bu}{{\boldsymbol u}}

\newcommand{\bx}{\boldsymbol x}
\newcommand{\by}{{\boldsymbol y}}

\newcommand{\bxf}{{\bx(\cdot)}}  

\newcommand{\buf}{{\bu(\cdot)}}  

\newcommand{\bP}{{\boldsymbol P}}

\newcommand{\bzero}{{\bf 0}}

\newcommand{\bnu}{{\mbox{\boldmath $\nu$}}}

\newcommand{\bpsi}{{\mbox{\boldmath $\psi$}}}
\newcommand{\blam}{{\mbox{\boldmath $\lambda$}}}

\newcommand{\bxi}{{\mbox{\boldmath $\xi$}}}

\newcommand{\bchi}{\mbox{\boldmath$\chi$}}

\newcommand{\bzeta}{\mbox{\boldmath$\zeta$}}
\newcommand{\etab}{\mbox{\boldmath$\eta$}}
\newcommand{\bsigma}{\mbox{\boldmath$\sigma$}}
\author{ %
I. M. Ross\thanks{Distinguished Professor, Department of Mechanical and Aerospace Engineering. 
}
\\
\textit{Naval Postgraduate School, Monterey, CA 93943}
}
\title{Transversality Conditions for Boundary Constraints Defined by Differential Equations}

\abstract{
What are the transversality conditions for an optimal control problem when the boundary conditions are defined by differential equations?  This seemingly bizarre question is motivated by trajectory optimization problems in the $N$-body system. The question, however, is more fundamental and goes beyond problems in astrodynamics to nonintegrable dynamical systems in general. The main contribution of this paper is the development of generic initial- and final-time transversality conditions for optimal control problems whose boundary conditions are defined in terms of differential equations with side conditions. The mathematical definition of differential boundary conditions are part of the foundations developed in this paper.  To support the new fundamentals, the concept of coordinated/uncoordinated clock times and weak adjoint covectors are introduced. In the case of uncoordinated clock times, the new transversality conditions reveal that there exists a special situation where a weak adjoint covector is orthogonal to the vector field of the boundary differential equation.  This condition is sharply different from the classical statement of orthogonality with respect to the endpoint manifold. The theorems developed in this paper are generic.  An application of the theorems to several cases in the three-body problem are described in separate papers. 
}

\begin{document}
\maketitle


\section{A Motivating Problem}
\label{sec:Intro}
We begin with a suite of practical astrodynamics problems for the sole purpose of providing a  context  for an unconventional mathematical formulation of boundary conditions. In terms of the specific objectives of this paper, it is not necessary for the reader to have any knowledge of astronautics.  With this perspective in mind, consider the problem of optimally transferring a spacecraft from an initial orbit to a final orbit in the restricted three-body problem\cite{Farq-1970,szebehely_theory_1967}.  Candidate initial and final orbits are depicted in Fig.~\ref{fig:boundaryOrbits}.
%
%
\begin{figure}[h!]
      \centering
      {\parbox{0.75\columnwidth}{
      \centering
      {\includegraphics[width = 0.75\columnwidth]{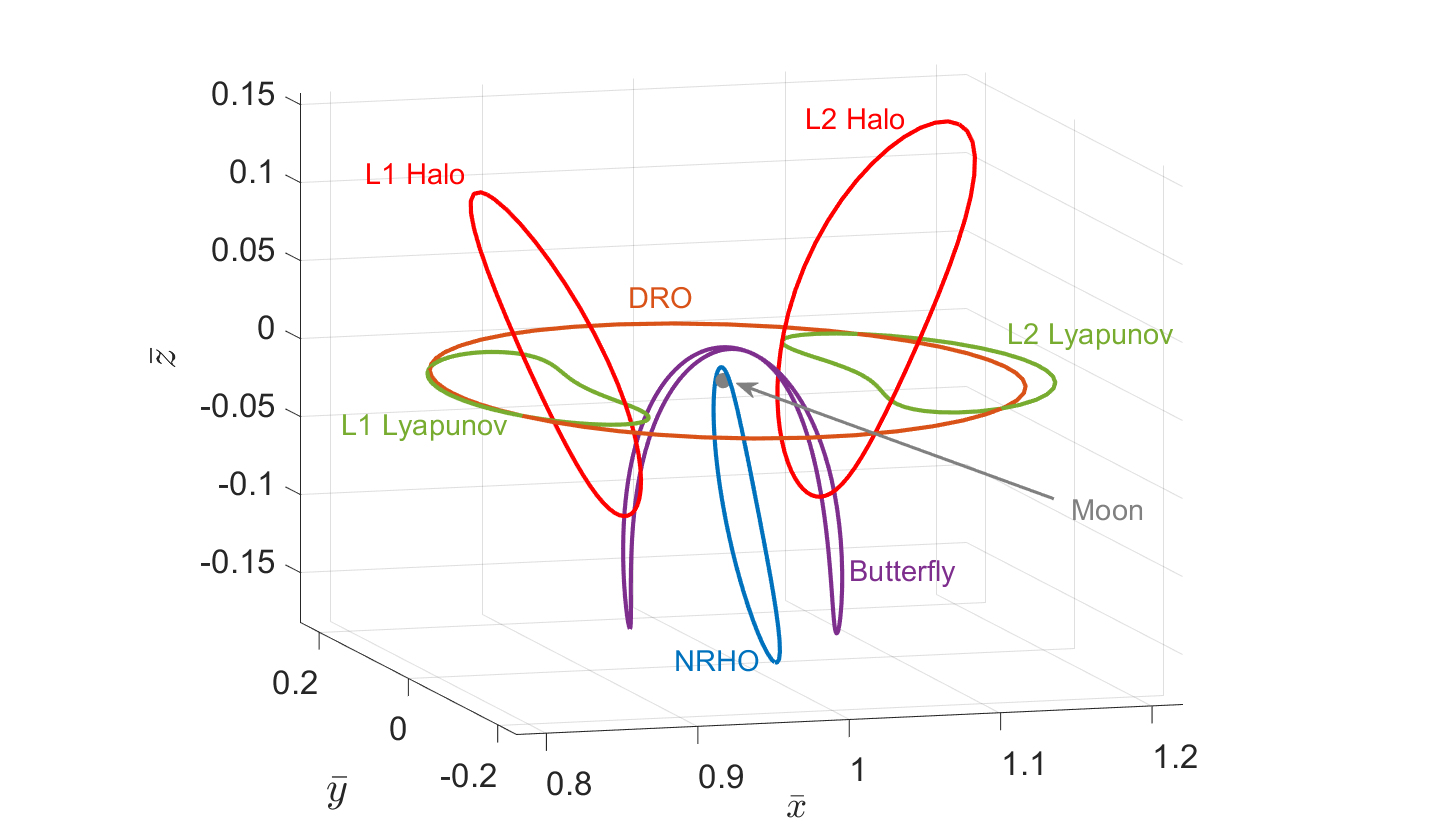}}
      \caption{{Practical examples of orbits that cannot be naturally defined in terms of algebraic equations (adapted from \cite{dixon-diss-2025}, with permission). }}\label{fig:boundaryOrbits}
      }
      }
\end{figure}
%
Obviously, this problem definition is superficially similar to a text-book two-body optimal control problem\cite{longuski,ross-book,brysonHo,vinter}. In the classic two-body problem, the initial- and final orbits are conics: circles, ellipses, hyperbolas etc.  Such geometric shapes can be described in terms of algebraic equations or so-called element sets\cite{curtis}. The element sets are integrals of motion. In the three-body problem, the orbits, such as the ones shown in Fig.~\ref{fig:boundaryOrbits}, are not describable in terms of similar integrals of motion.  In fact, the prior sentence is essentially a statement of the famous Bruns-Poincar\'{e} theorem of impossibility\cite{ams-1913,phd-1993,moser-1971}. Hence, it follows that a major difference between a three-body and a two-body trajectory optimization problem is that the initial and final orbits in the former case cannot be naturally described in terms of algebraic (including transcendental\cite{ams-1913}) equations. Nonetheless, the orbits are indeed computable as evident in Fig.~\ref{fig:boundaryOrbits}.  Ignoring the details of this computational process, the theoretical concept for defining an inital/final orbit for a trajectory optimization problem in this setting is equivalent to imposing certain special conditions on the evolution of a (three-body) dynamical system given generically by,
\begin{equation}\label{eq:ode-g}
\dot\bx = \bg(\bx, t)
\end{equation}
where, $\bg: \real{6} \times \Real \to \real{6}$ is a time-varying vector field defined by the equations of motion\cite{szebehely_theory_1967}.  Thus, for example, the L1 Lyapunov orbit shown in Fig.~\ref{fig:boundaryOrbits} was obtained by a numerical propagation of the initial conditions in an elliptic, restricted, three-body problem given approximately (see \cite{dixon-diss-2025} for more accurate numbers) by,
\begin{equation}\label{eq:intro-init-0}
\bx(t^\sharp_0) = (0.803, 0.000, 0.000, 0.317, 0.000, 0.000, 0.000)
\end{equation}
where $t^\sharp_0$ corresponds to time at perilune.  The main point to note here is that if any of the orbits shown in Fig.~\ref{fig:boundaryOrbits} are designated as initial orbits for a trajectory optimization problem, then none of them are given by some algebraic equation of the form\cite{ross-book},
\begin{equation}\label{eq:e0=0}
\be_0(\bx_0, t_0) = \bzero
\end{equation}
where $\be_0: \real{6} \times \Real \to \real{N_{e_0}}, \ N_{e_0} \ge 1 $ is an explicitly given algebraic function.  Typical boundary conditions in optimal control are customarily assumed to be given according to \eqref{eq:e0=0}\cite{brysonHo,longuski,ross-book}.  See also Section~\ref{sec:reviewTVC} for further details.  It is extremely important to note that the pair $(\bx_0, t_0) \in \real{6} \times \Real$ in the context of an initial L1 Lyapunov orbit shown in Fig.~\ref{fig:boundaryOrbits}, is not  the specific point given by \eqref{eq:intro-init-0}; rather, it is any point on that orbit. Likewise, the near rectilinear halo orbit (NRHO) shown in Fig.~\ref{fig:boundaryOrbits} was obtained by a numerical propagation of the initial condition given approximately by\cite{dixon-diss-2025},
\begin{equation}\label{eq:intro-init-f}
\bx(t^\natural_f) = (1.026, 0.000, -0.194, 0.000, -0.108, 0.000, 0.000)
\end{equation}
Needless to say, if the target manifold were to be the NRHO shown in Fig.~\ref{fig:boundaryOrbits}, it is not given in terms of a final-time algebraic equation,
\begin{equation}\label{eq:ef=0}
\be_f(\bx_f, t_f) = \bzero
\end{equation}
where, $\be_f$ is defined similarly to $\be_0$\cite{ross-book}. In the context of Fig.~\ref{fig:boundaryOrbits}, the pair $(\bx_f, t_f) \in \real{6} \times \Real$ is any point on the final NRHO and not necessarily the specific point given by \eqref{eq:intro-init-f}. Thus, there is a gap, or at least the appearance of a gap, in the existing theory of optimal control in both the mathematical formulation of boundary conditions as well as the ensuing transversality conditions.  The mathematical details of this gap are reviewed in Section~\ref{sec:reviewTVC}.

Before proceeding further, it is useful to observe the following points in the suite of three-body orbit transfer problems implied in Fig.~\ref{fig:boundaryOrbits}:
\begin{enumerate}
\item The rudimentary problem posed in the previous paragraphs is not to transfer a spacecraft from a given initial point, $\bx(t_0^\sharp)$, to a final point, $\bx(t_f^\sharp)$.  The transversality conditions for given/known point conditions are already well-established\cite{brysonHo,longuski,ross-book}.
\item For the purposes of the theory developed in this paper, we are not interested in solving the computational problem by curve-fitting the initial and final orbits/manifolds as done in Ref.~\cite{dixon-er3bp-1}, for example. Curve-fitting will solve the computational problem (in principle) but will not expose the more fundamental mathematics underlying the case of boundary conditions defined by differential equations.
\item The motivating application problem posed herein is to find the necessary conditions for the optimal departure point, $\bx_0$, from the initial orbit/manifold as well as the necessary conditions for the optimal arrival point, $\bx_f$, at the target orbit/manifold.  In the elliptic, restricted, three-body problem and in the general $N$-body problem, these departure and arrival points are also dependent on the clock times, $t_0$ and $t_f$.
\item Although we are motivated by the problems of the type shown in Fig.~\ref{fig:boundaryOrbits}, our objective is neither limited to the three-body problem nor specific to spacecraft trajectory optimization problems. Our objective is primarily to derive transversality conditions for any problem, space-related or otherwise, that may have similar stipulations on boundary conditions.
 
\end{enumerate}
By answering the fundamental question posed in this paper, we arrive at several computable transversality conditions.  The details of the numerics for a particular three-body application problem selected from the possible suite of boundary conditions depicted in Fig.~\ref{fig:boundaryOrbits} are described in \cite{dixon-diss-2025}.

\section{A Brief Review of Existing Transversality Conditions}
\label{sec:reviewTVC}

It is possible to state the transversality conditions for an optimal control problem in terms of a simple statement\cite{vinter,clarke-2013book} that is more general than Pontryagin's original theorem\cite{PBGM-1962}.  Although a precise statement of these generalized transversality conditions requires the tools of nonsmooth calculus\cite{vinter,clarke-classic-1990}, we will only use them in a smooth setting because the latter requires substantially less mathematical machinery than the former.  To this end, we pose the following optimal control problem:  Find the tuple $\big( \bx(\cdot), \bu(\cdot), t_0, t_f, \bp \big)$ that solves Problem~$(\bP)$ defined by:
%
\begin{eqnarray}
&\bx(t) \in \real{N_x}, \quad \bu(t) \in \real{N_u}, \quad t \in \Real, \quad \bp \in \real{N_p}   & \nonumber\\
& (\textsf{$\bP$}) \left\{
\begin{array}{crl}
\displaystyle\mathop{\emph{Minimize }}_{[\bxf, \buf, t_0, t_f, \bp]} & J[\bx(\cdot), \bu(\cdot), t_0, t_f, \bp] =& E(\bx(t_0), t_0, \bx(t_f), t_f, \bp)\\[.5em]
& & \quad \displaystyle + \int_{t_0}^{t_f} F(\bx(t), \bu(t), t, \bp)\, dt \\[1em]
\emph{Subject to}& \dot\bx(t) =& \bff(\bx(t), \bu(t), t, \bp)  \\
& \bu(t) \in & U(t) \\
&  \big(\bx(t_0), t_0, \bx(t_f), t_f, \bp\big) \in &  S \\
\end{array} \right. & \label{eq:probPbold}
\end{eqnarray}
%
The notation used in \eqref{eq:probPbold} is as follows\cite{ross-book}:
\begin{enumerate}
\item $N_{(\cdot)} \in \mathbb{N}$ represents the number of variables or functions for the symbols denoted by the subscript;
\item $\bx \in \real{N_x}$ is the state variable, $\bu \in \real{N_u}$ is the control variable, $t \in \Real$ is the clock time, $t_0 \in \Real$ is an initial clock time, $t_f > t_0$ is a final clock time, and $\bp \in \real{N_p}$ is an optimization parameter;
\item $(\bxf, \buf)$ is the state-control function pair (i.e., an element in some Sobolev space\cite{ross-book,vinter,clarke-2013book});
\item $J: \big(\bx(\cdot), \bu(\cdot), t_0, t_f, \bp\big) \mapsto \Real $ is the standard (``Bolza'') cost functional, $E: \real{N_x} \times \Real \times \real{N_x} \times \Real \times \real{N_p} \to \Real $ is the endpoint (``Mayer'') cost and $F: \real{N_x} \times \real{N_u} \times \Real \times \real{N_p} \to \Real$ is the running (``Lagrange'') cost\cite{longuski,brysonHo,ross-book};
\item $\bff$ is a non-autonomous (i.e., time-dependent) dynamics function, $\bff:\real{N_x} \times \real{N_u} \times \Real \times \real{N_p} \to \real{N_x} $;
\item $U(t)$ is a given compact set (which may vary with time) that defines the (instantaneous) allowable values of $\bu(t)$;
\item $S \subset \real{N_x} \times \Real \times \real{N_x} \times \Real \times \real{N_p}$ is a given closed set, whose parametrization is part of the main topic of this paper. See Fig.~\ref{fig:Examples4S} for illustrative sets. 
\end{enumerate}
%
\begin{figure}[h!]
      \centering
      {\parbox{0.9\columnwidth}{
      \centering
      {\includegraphics[width = 0.6\columnwidth]{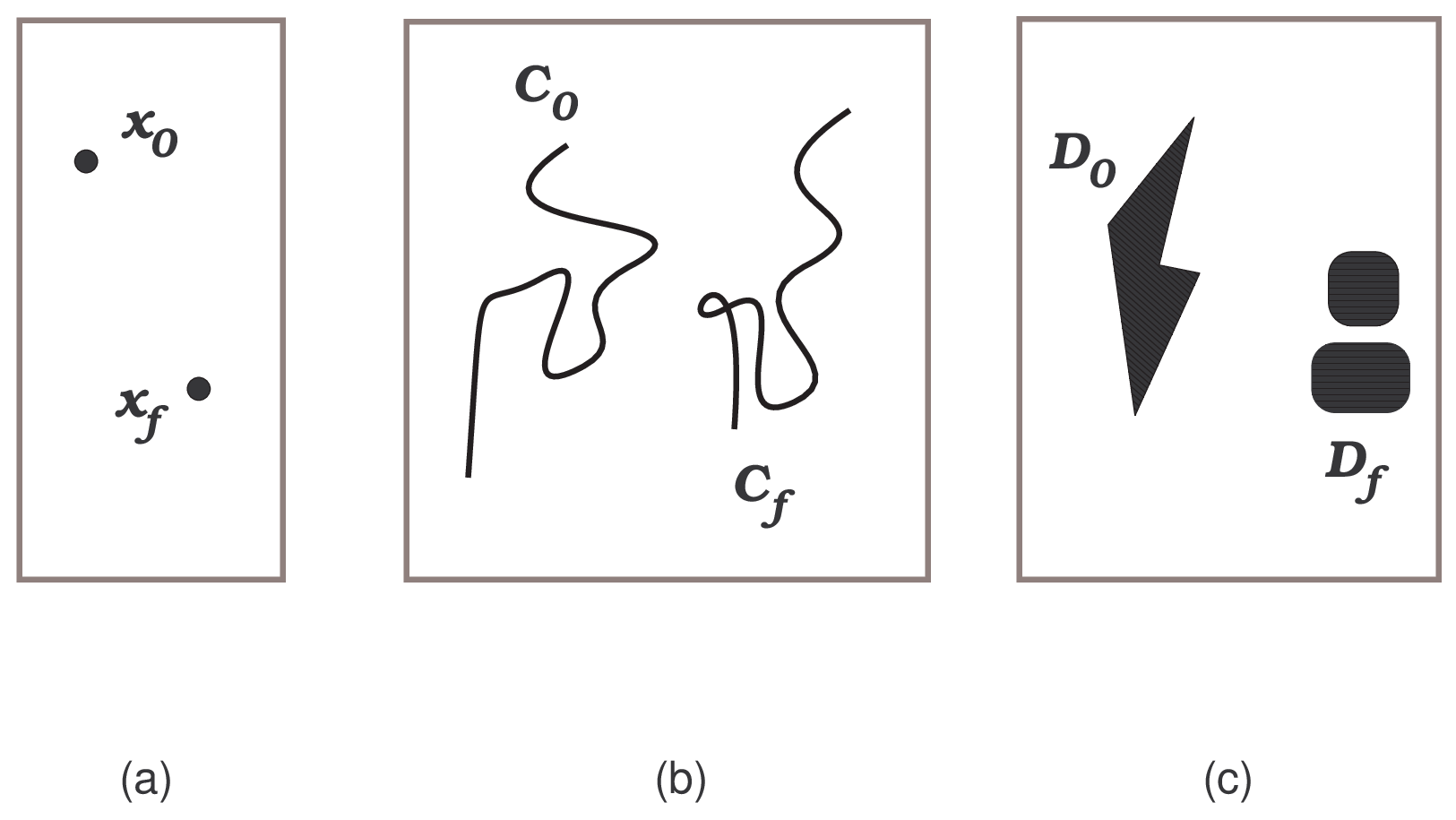}}
      \caption{{Illustrative examples of the set $S$ used in the definition of Problem~$(\bP)$ given by \eqref{eq:probPbold}.}}\label{fig:Examples4S}
      }
      }
\end{figure}
%
The set $S$ in Problem~$(\bP)$ is quite arbitrary.  In the simplest case of fixed points and fixed times (with the optimization parameter $\bp$ fixed) $S$ is a singleton.  Frequently, $S$ is a Cartesian product of two sets comprising an initial and final set. These aspects are illustrated in Fig.~\ref{fig:Examples4S}.  In panel (a), the set $S$ is just (the Cartesian product of) two points.  In panel (b), $S$ comprises curves denoted by $C_0$ and $C_f$ while panel (c) illustrates two nonconvex sets $D_0$ and $D_f$.   
%
\begin{remark}
Path constraints and functional constraints\cite{ross-book} may be incorporated in the definition of Problem~$(\bP)$. Because the main topic of this paper is transversality conditions, these additional constraints do not alter the results to follow.  Hence, there is no loss in generality in omitting these constraints in the statement of Problem~$(\bP)$.
\end{remark}
We assume all data functions in Problem~$(\bP)$ to be continuously differentiable with respect to all of their respective arguments.  Hence, a statement of the necessary conditions for  Problem~$(\bP)$ involves the structure and properties of the set $S$.  This topic constitutes the main discussion point of this section.

\subsection{Mathematical Preliminaries}
Because the endpoint condition in Problem~$(\bP)$ is defined in terms of an abstract set, it should not be a surprise that the transversality conditions are stated in terms of some property of $S$.  This property is called the normal cone\cite{vinter,clarke-2013book} and is defined in terms of a vector that ``leaves $S$  orthogonally\cite{clarke-2013book}.''  Although it is possible to define a normal cone independent of a tangent cone\cite{vinter,clarke-2013book}, we take the simpler path of defining a tangent cone first.
Fig.~\ref{fig:pacman} provides a useful visual in understanding the definitions to follow.
%
\begin{definition}[Tangent Cone\cite{clarke-2013book,NW:NumOptBook,Jahn-2007}]\label{definition:TC}
A vector $\bd_\chi$ is said to be a tangent vector to $S$ at the point $\bchi$ if there is a sequence of feasible vectors $\set{\bchi^k}$ and a sequence of positive scalars $\set{s^k}$ with $s^k \to 0$ such that %
$$\lim_{k \to \infty} \frac{\bchi^k - \bchi}{s^k} = \bd_\chi $$
The set of all tangent vectors at $\bchi$, denoted by $T_S(\bchi)$, is called the tangent cone (to $S$ at $\bchi$).
\end{definition}
\begin{definition}[Normal Cone\cite{vinter,clarke-2013book,NW:NumOptBook}]\label{definition:NC}
The normal cone to $S$ at a point $\bchi \in S$, denoted by $N_S(\bchi)$, is defined as,
$$N_S(\bchi):= \mathop{\lim\sup}_{k \to \infty} T_S^*(\bchi_k)$$
where $T_S^*(\bchi_k)$, called the polar cone to $T_S(\bchi_k)$, is defined by,
\begin{equation}\label{eq:polarconedef}
T_S^*(\bchi_k):= \set{\bzeta_\chi:\ \bzeta_\chi^T\bd_\chi \le 0 \quad \forall\ \bd_\chi \in T_S(\bchi_k) }
\end{equation}
A vector $\bzeta_\chi \in N_S(\bchi)$ is called a normal vector to $S$ at $\bchi$.
\end{definition}
\begin{figure}[h!]
      \centering
      {\parbox{0.9\columnwidth}{
      \centering
      {\includegraphics[width = 0.6\columnwidth]{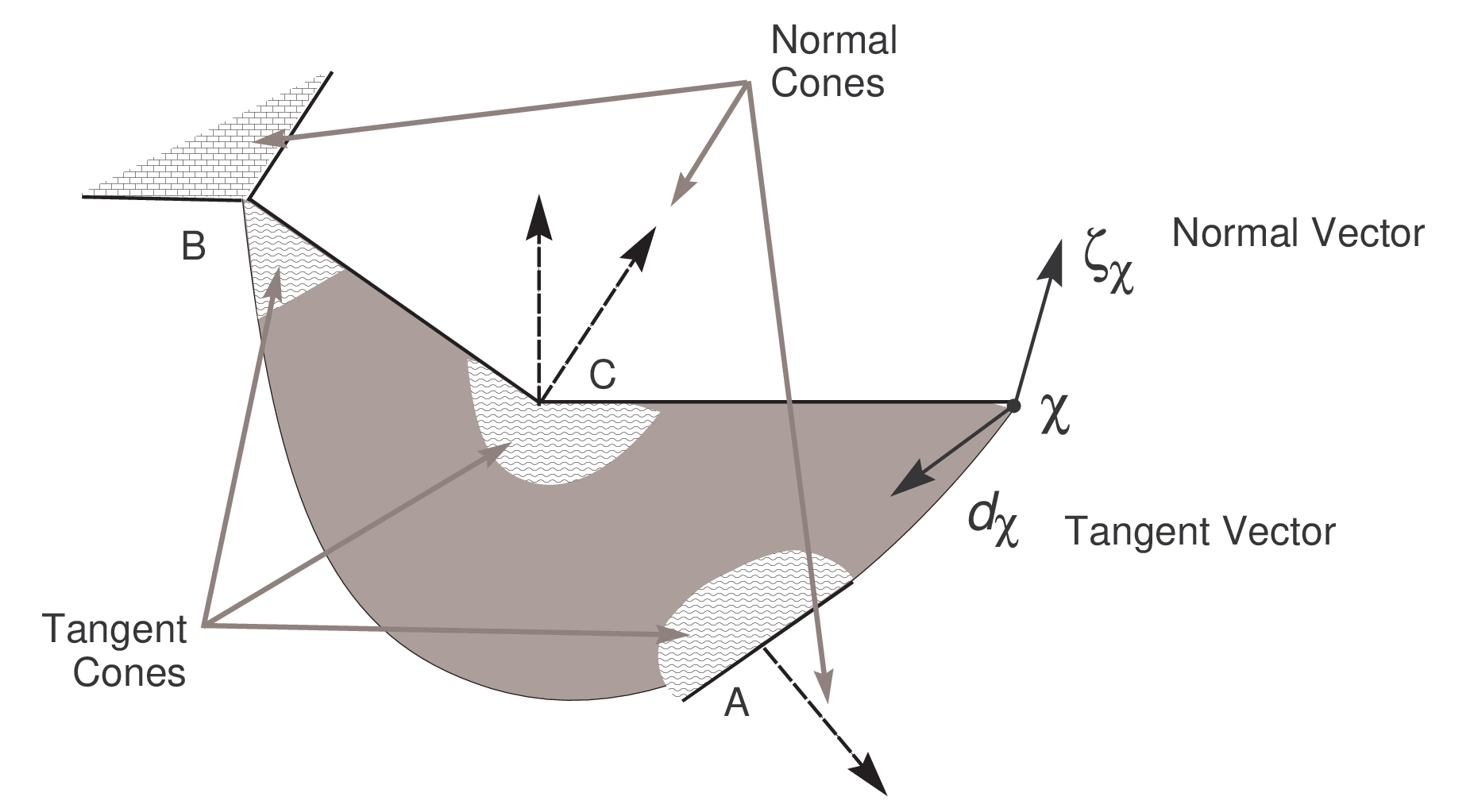}}
      \caption{{Tangent and normal vectors and cones at key points of a nonconvex set $S$.}}\label{fig:pacman}
      }
      }
\end{figure}
%
Fig.~\ref{fig:pacman} illustrates all these cones at key points labeled $A$, $B$ and $C$:
\begin{enumerate}
\item Consider the point $A$ that lies on a smooth boundary of a set. Then the normal cone, $N_S(A)$, is just the outward perpendicular to the tangent hyperplane.  The tangent cone is the half-space indicated at the point $A$ in Fig.~\ref{fig:pacman}.
\item Consider the point $B$ that lies on a nonsmooth boundary point.  This point is locally convex. In this case, the (convex) set of all outward normals is the normal cone $N_S(B)$;
\item Consider the point $C$ on a nonsmooth boundary point that is not locally convex.  In this case, the limiting set of all outward normals is the (nonconvex) normal cone $N_S(C)$ that comprises just the two vectors shown in Fig.~\ref{fig:pacman}.
\end{enumerate}
In the literature on nonsmooth calculus, a large number of tangent and normal cones are defined\cite{vinter,clarke-2013book,sussmann-cones-2005,benko-cones-2019}.  The tangent cone of Definition~\ref{definition:TC} is called the Bouligand or contingent tangent cone\cite{benko-cones-2019,clarke-2013book,Jahn-2007} and the normal cone defined in Definition~\ref{definition:NC} is called the Mordukhovich or limiting normal cone\cite{sussmann-cones-2005,clarke-2013book,vinter}.  In this paper, we will not use these qualifiers and rely exclusively on Definitions~\ref{definition:TC} and \ref{definition:NC} because these cones are the most relevant to the formulation of transversality conditions.
%
\begin{remark}\label{rem:ortho}
In general, the tangent and normal cones are nonconvex; see the cones illustrated at the point $C$ in Fig.~\ref{fig:pacman}.  Furthermore, note the following:
\begin{enumerate}
\item From their definitions, it is clear that zero is an element of both the tangent and normal cones.  They are sketched in Fig.~\ref{fig:pacman} as emanating from the point $\bchi$ as a floating vector\cite{rockwets} instead of the origin\cite{Jahn-2007} as in Fig.~\ref{fig:TandNcones}.
%
\begin{figure}[h!]
      \centering
      {\parbox{0.9\columnwidth}{
      \centering
      {\includegraphics[width = 0.55\columnwidth]{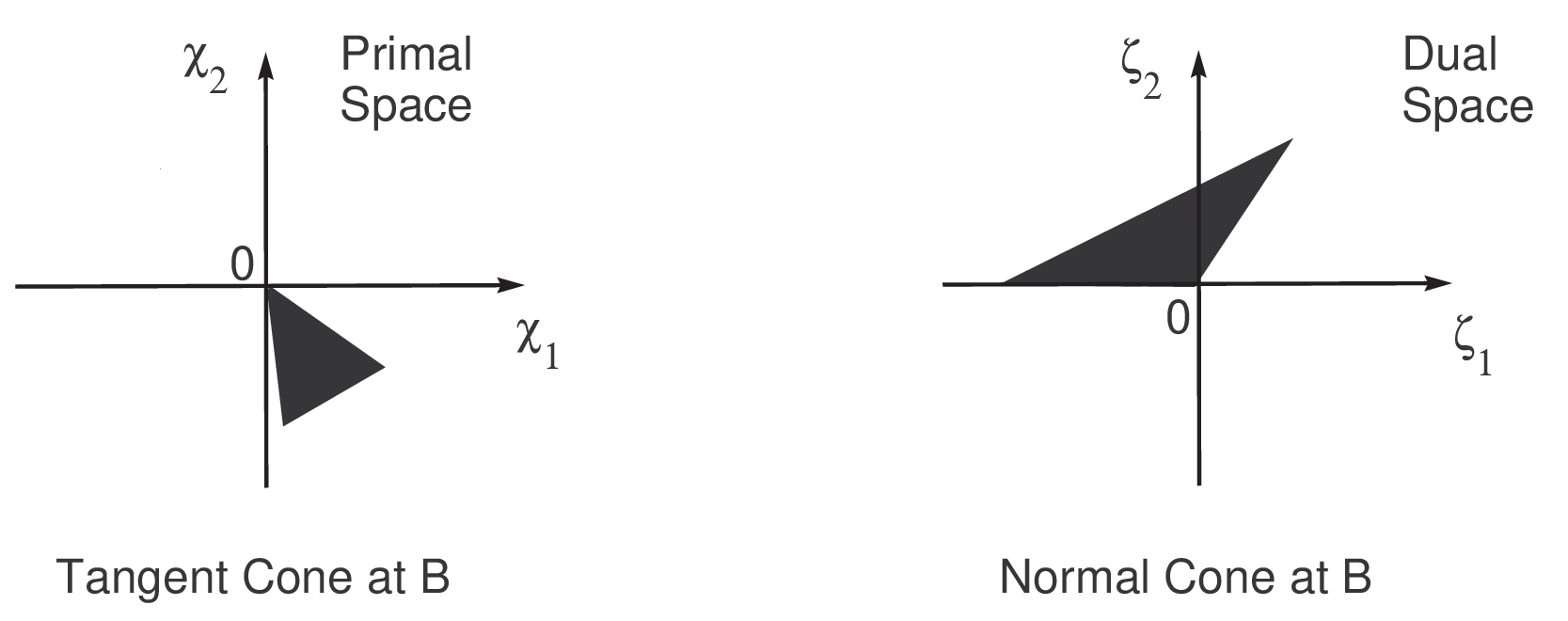}}
      \caption{{Tangent and normal cones at the point $B$ of Fig.~\ref{fig:pacman} and their different spaces.}}\label{fig:TandNcones}
      }
      }
\end{figure}
\item The normal cone is a construct in dual space.  The vector $\bzeta_\chi$ in \eqref{eq:polarconedef} is a covector whose ``units'' are in ``frequency''\cite{ross-book}.  For example, the units of $(\chi_1, \chi_2)$ and $(\zeta_1, \zeta_2)$ in Fig.~\ref{fig:TandNcones} are not the same.  See \cite{ross-book} for details pertaining to a ``physical interpretation'' of covectors.  The normal cone sketched in Fig.~\ref{fig:pacman} uses the notion of a proxy vector to the covector\cite{ross-book} in order to display its connection to primal space.
\item Even though a vector and covector are elements in different vector spaces as shown in Fig.~\ref{fig:TandNcones}, we say $\bzeta_\chi$ is ``orthogonal'' to  $\bd_\chi$ if $\bzeta_\chi^T \bd_\chi = 0$.  Note that each component of $\bzeta_\chi$ can have different units\cite{ross-book}.
\end{enumerate}
\end{remark}

The transversality conditions for Problem~$(\bP)$ can be stated in terms of the following three constructs:
\begin{enumerate}
\item[(a)] the endpoint values of the adjoint covector (costate),
\item[(b)] the endpoint values of the lower (i.e., minimized) Hamiltonian, and
\item[(c)] the normal cone to the set $S$ that defines the endpoint constraints.
\end{enumerate}
To state these conditions we first write the (Pontryagin) Hamiltonian function for the pair $(F, \bff)$ as\cite{ross-book},
\begin{equation}\label{eq:H=pontry}
H(\blam, \bx, \bu, t, \bp) := \nu^0\, F(\bx, \bu, t, \bp) + \blam^T\bff(\bx, \bu, t, \bp)
\end{equation}
where $\nu^0 \ge 0$ is a cost multiplier and $\blam \in \real{N_x}$ is a covector that satisfies the adjoint differential equation,
\begin{equation}\label{eq:adjoint}
-\dot\blam(t) = \partial_{\bx} H(\blam(t), \bx(t), \bu(t), t, \bp)
\end{equation}
Typically, we have $\nu^0 > 0$, the so-called normal case\cite{ross-book}.  If $\nu^0 = 0$ the case is deemed ``abnormal.'' In the normal case, we may use $\nu^0 > 0$ to scale and balance the adjoint equations\cite{scaling} in a computational setting or simply set $\nu^0 = 1$ for theoretical analysis. The centerpiece of Pontryagin's Principle is the Hamiltonian Minimization Condition (HMC), a static, time-varying, instantaneous optimization problem, defined by\cite{ross-book}:
%
%
\begin{eqnarray}
& (\textsf{$HMC$}) \left\{
\begin{array}{lll}
\displaystyle\mathop{\emph{Minimize }}_{\bu(t)} & H(\blam(t), \bx(t), \bu(t), t, \bp)\\[.5em]
\emph{Subject to}
& \bu(t) \in  U(t)
\end{array} \right. & \label{eq:probHMC}
\end{eqnarray}
%
A (conceptual) solution to Problem~$(HMC)$ is given by,
\begin{equation}\label{eq:u=argmin}
\bu^*(\blam(t), \bx(t), t, \bp) = \arg\min_{\bu(t) \in U(t)} H(\blam(t), \bx(t), \bu(t), t, \bp)
\end{equation}
Substituting \eqref{eq:u=argmin} in \eqref{eq:H=pontry} eliminates $\bu$ and generates the lower Hamiltonian\cite{ross-book,clarke-2013book},
\begin{equation}
\mathcal{H}(\blam(t), \bx(t),  t, \bp) = H(\blam(t), \bx, \bu^*(\blam(t), \bx(t), t, \bp), t, \bp)
\end{equation}
Note that the lower Hamiltonian $\mathcal{H}$ is not a function of the control variable. It is frequently nondifferentiable\cite{ross-book}.

In the discussions to follow, it will be convenient to use the following abbreviation:
\begin{equation}\label{eq:x0xf-abbrv}
\bx(t_0) \equiv \bx_0, \quad\quad \bx(t_f) \equiv \bx_f, \quad{and}\quad \mathcal{H}(\blam(t), \bx(t),  t, \bp) \equiv \mathcal{H}[@t]
\end{equation}

\subsection{Transversality Conditions for the Case When $\bp$ is Not an Optimization Variable}
Much of the literature contains necessary conditions for the case when $\bp$ is not an optimization variable.  Hence, the case when $\bp$ is indeed an optimization variable deserves special consideration.  This case is deferred to the next subsection.
\begin{theorem}[\cite{clarke-2013book,vinter}]\label{theorem-TVC-1}
Let Problem~($\bP^\sharp$) denote the special case of Problem~$(\bP)$ when $\bp$ is fixed to a given value $\bp = \bp^\sharp$. Let $\mathcal{H}[@t]$ denote an abbreviated notation for $\mathcal{H}(\blam(t), \bx(t),  t, \bp^\sharp)$. Then the transversality condition for Problem~($\bP^\sharp$) is given by,
\begin{equation}\label{eq:TVC-abstract-no-p}
\big(-\blam(t_0), \mathcal{H}[@t_0], \blam(t_f), -\mathcal{H}[@t_f]\big) \in \nu^0\,\partial_{(\bx_0, t_0, \bx_f, t_f)} E(\bx_0,t_0, \bx_f, t_f, \bp^\sharp) +  N_S(\bx_0, t_0, \bx_f, t_f; \bp^\sharp)
\end{equation}
where $\bx_0$ and $\bx_f$ are $\bx(t_0)$ and $\bx(t_f)$ as respectively in accordance with \eqref{eq:x0xf-abbrv}.
\end{theorem}
%
Theorem~\ref{theorem-TVC-1} is an elegant statement of transversality conditions cast in geometric terms. However, for \eqref{eq:TVC-abstract-no-p} to be useful from a computational perspective, it requires a mechanism to compute $N_S$.  To clarify this statement,  we first rewrite \eqref{eq:TVC-abstract-no-p} in a more convenient ``equality'' form as,
\begin{equation}\label{eq:TVC-abstract-no-p=}
\big(-\blam(t_0), \mathcal{H}[@t_0], \blam(t_f),  -\mathcal{H}[@t_f] \big) = \nu^0\,\partial_{(\bx_0, t_0, \bx_f, t_f)} E(\bx_0, t_0, \bx_f, t_f, \bp^\sharp) + \etab
\end{equation}
for some $\etab \in N_S(\bx_0, t_0, \bx_f, t_f, \bp^\sharp)$.  Equation~(\ref{eq:TVC-abstract-no-p=}) reveals that the main problem in computing the transversality condition is centered around finding a suitable $\etab$.

As powerful as it is, it turns out Theorem~\ref{theorem-TVC-1} is not sufficient to handle the production of transversality conditions for the motivating problems illustrated in Fig.~\ref{fig:boundaryOrbits}.  In fact, what is needed is an inclusion of the case when $\bp$ is an optimization variable.

\subsection{Transversality Conditions for the General Case}
The general case when $\bp$ is an optimization variable deserves special attention due to an unfortunate error in the original work by Pontryagin et al\cite{PBGM-1962}.  This error in the necessary optimality condition for $\bp$ was first noted and corrected by Hofer and Sagirow in 1968\cite{HS-1968}.  The corrected necessary condition was subsequently generalized by other authors\cite{AG-1973,kyber-80,kyber-81}.  In principle, the necessary condition for an optimal $\bp$ is not part of the transversality condition; however, because $\bp$ appears in the endpoint functional and constraints (see \eqref{eq:probPbold}) it cannot be decoupled in the computation of the normal cone in \eqref{eq:TVC-abstract-no-p} (or \eqref{eq:TVC-abstract-no-p=}).  As noted earlier, we will utilize the general case in Section~\ref{sec:newTVCs}.

To combine the result of Theorem~\ref{theorem-TVC-1} with the results of \cite{HS-1968,AG-1973,kyber-80,kyber-81} we define a quantity $\bsigma \in \real{N_p}$ according to,
\begin{equation}\label{eq:sigma=bydef}
\bsigma:=  \int_{t_0}^{t_f} \partial_\bp H(\blam(t), \bx(t), \bu(t), t, \bp)\, dt
\end{equation}
If $\bp$ is unconstrained, then the ``corrected'' necessary condition\cite{HS-1968,AG-1973,kyber-80,kyber-81} for optimality is given by,
\begin{equation}
-\bsigma = \nu^0\,\partial_\bp E(\bx_0, t_0, \bx_f, t_f, \bp)
\end{equation}
If $\bp$ is constrained then $\bsigma + \nu^0\,\partial_\bp E(\bx_0, t_0, \bx_f, t_f, \bp)$ satisfies necessary conditions\cite{AG-1973,kyber-80,kyber-81} given by the Karush-Kuhn-Tucker (KKT) theorem\cite{NW:NumOptBook}.  In view of this result, Theorem~\ref{theorem-TVC-1} generalizes accordingly and can be stated as follows:
\begin{theorem}\label{theorem-TVC-2}
The transversality condition for Problem~($\bP$) is given by,
\begin{equation}\label{eq:TVC-abstract-yes-p}
\big(-\blam(t_0), \mathcal{H}[@t_0], \blam(t_f),  -\mathcal{H}[@t_f], -\bsigma \big) \in \nu^0\,\partial_{(\bx_0, t_0, \bx_f, t_f, \bp)} E(\bx_0,t_0, \bx_f, t_f, \bp) +  N_S(\bx_0, t_0, \bx_f, t_f, \bp)
\end{equation}
or equivalently as,
\begin{equation}\label{eq:TVC-abstract-yes-p=}
\big(-\blam(t_0), \mathcal{H}[@t_0], \blam(t_f), -\mathcal{H}[@t_f], -\bsigma \big) = \nu^0\,\partial_{(\bx_0, t_0, \bx_f, t_f, \bp)} E(\bx_0, t_0, \bx_f, t_f, \bp) + \etab
\end{equation}
where $\etab \in N_S(\bx_0, t_0, \bx_f, t_f, \bp)$ and $\bx_0$,  $\bx_f$, $\mathcal{H}[@t_0]$ and $\mathcal{H}[@t_f]$  denote $\bx(t_0)$,  $\bx(t_f)$, $\mathcal{H}(\blam(t_0), \bx(t_0),  t_0, \bp)$, and $\mathcal{H}(\blam(t_f), \bx(t_f),  t_f, \bp)$ respectively.
\end{theorem}
%
In the sections to follow, we will use Theorem~\ref{theorem-TVC-2} as our building block to formulate new computable transversality conditions.

\subsection{Transversality Conditions When $S$ Is Parameterized by Algebraic Inequalities }
The most common ``application'' of Theorem~\ref{theorem-TVC-2} (and hence Theorem~\ref{theorem-TVC-1}) is when the set $S$ is parameterized by algebraic inequalities,
\begin{equation}\label{eq:S=efun}
S = S^{alg} := \set{(\bx_0, t_0, \bx_f, t_f, \bp):\ \be^L \le \be(\bx_0, t_0, \bx_f, t_f, \bp) \le \be^U}
\end{equation}
where $\be: (\bx_0, t_0, \bx_f, t_f, \bp) \mapsto \real{N_e}$ is a continuously differentiable (``algebraic'') function, $\be^L \in \real{N_e}$ is a lower bound on the allowable values of $\be$ and $\be^U \ge \be^L$ is the corresponding upper bound.  A large class of practical optimal control problems (barring the ones identified in Section~\ref{sec:Intro}) have boundary conditions that are conveniently described in the form given by \eqref{eq:S=efun}.  Then, it can be shown\cite{vinter,clarke-2013book}
that Theorem~\ref{theorem-TVC-2} reduces to the transversality conditions given by the following familiar theorem\cite{ross-book}:
%
\begin{theorem}\label{theorem:TVC-alg}
Let Problem~($\bP^{alg}$) denote the special case of Problem~$(\bP)$ where the boundary condition is defined by $S^{alg}$ given by \eqref{eq:S=efun}.  Then the transversality conditions for Problem~$(\bP^{alg})$ together with the necessary condition for optimal $\bp$ are given by,
\begin{subequations}\label{eq:TVC-alg}
\begin{align}
-\blam(t_0) &= \partial_{\bx_0} \overline{E}(\bnu, \bx_0, t_0, \bx_f, t_f, \bp) \\
\mathcal{H}[@ t_0] &=  \partial_{t_0} \overline{E}(\bnu, \bx_0, t_0, \bx_f, t_f, \bp) \\
\blam(t_f) &= \partial_{\bx_f} \overline{E}(\bnu, \bx_0, t_0, \bx_f, t_f, \bp) \\
-\mathcal{H}[@ t_f] &=  \partial_{t_f} \overline{E}(\bnu, \bx_0, t_0, \bx_f, t_f, \bp) \\
 -  \bsigma & = \partial_{\bp}\overline{E}(\bnu, \bx_0, t_0, \bx_f, t_f, \bp)\label{eq:notTVC}
\end{align}
\end{subequations}
where $\overline{E}$ is the endpoint Lagrangian\cite{ross-book} defined by,
\begin{equation}\label{eq:Ebar=bydef}
\overline{E}(\bnu, \bx_0, t_0, \bx_f, t_f, \bp):= \nu^0\,E(\bx_0, t_0, \bx_f, t_f, \bp) + \bnu^T \be( \bx_0, t_0, \bx_f, t_f, \bp)
\end{equation}
and $\bnu \in \real{N_e}$ is an endpoint covector that satisfies the following complementarity conditions\cite{ross-book},
\begin{align}\label{eq:complementarity}
\bnu \dagger \be( \bx_0, t_0, \bx_f, t_f, \bp) \Leftrightarrow \nu_i \left\{
             \begin{array}{ll}
             \le 0, & \hbox{\text{if} }\  e_i( \bx_0, t_0, \bx_f, t_f, \bp) = e_i^L \\
             = 0, & \hbox{if }\  e_i^L < e_i( \bx_0, t_0, \bx_f, t_f, \bp) < e_i^U \\
             \ge 0, & \hbox{if }\  e_i( \bx_0, t_0, \bx_f, t_f, \bp) = e_i^U \\
             \text{any value}, & \hbox{if }\  e_i^L = e_i^U
                                                 \end{array}
                                               \right.
\end{align}
\end{theorem}
\begin{remark}\label{remark:eta-nu-connection}
Comparing \eqref{eq:TVC-alg} with \eqref{eq:TVC-abstract-yes-p=} it follows that $\etab$ (i.e., a vector in the normal cone; see Definition~\ref{definition:NC}) encapsulates, in one fell swoop, the combined notion of complementarity conditions (see \eqref{eq:complementarity}) and the idea of expressing a vector as a linear combination of the active gradients of the individual endpoint constraint functions $e_i(\bx_0, t_0, \bx_f, t_f, \bp), \ i = 1, \ldots, N_e$ according to the following:
\begin{multline}\label{eq:eta-meaning}
\etab \in N_{S^{alg}} :=  \left\{ \bxi\in \real{N_x} \times \Real \times \real{N_x} \times \Real \times \real{N_p}:\ \right.\\
\left. \bxi = \big[\partial_{(\bx_0, t_0, \bx_f, t_f, \bp)} \be( \bx_0, t_0, \bx_f, t_f, \bp)\big]^T \bnu,\   \bnu \dagger \be( \bx_0, t_0, \bx_f, t_f, \bp)  \right\}
\end{multline}
\end{remark}
\begin{remark}
If the endpoint constraint function is given by a box constraint, then $\bxi$ in \eqref{eq:eta-meaning} reduces to $\bnu$.  In this case, $\bnu$ can be visualized as a vector in the normal cone to a point in a box as illustrated in Fig.~\ref{fig:boxNC}.  Hence, when we encounter this special case, we may use the symbols $\etab$ and $\bnu$ interchangeably.
\end{remark}
%
\begin{figure}[h!]
      \centering
      {\parbox{0.9\columnwidth}{
      \centering
      {\includegraphics[width = 0.5\columnwidth]{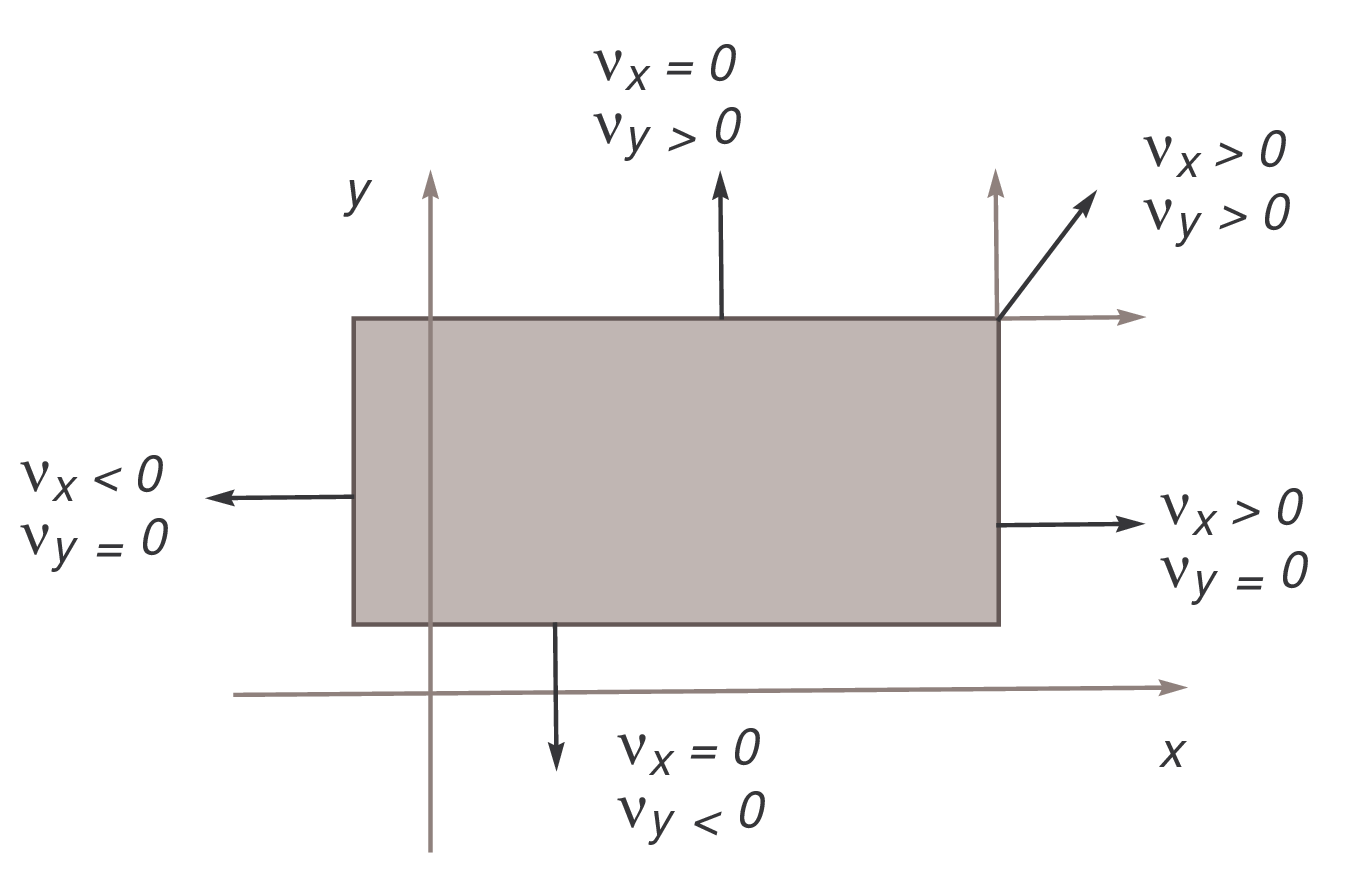}}
      \caption{{Normal cones and vectors at various points on a box constraint.}}\label{fig:boxNC}
      }
      }
\end{figure}

Theorem~\ref{theorem:TVC-alg} is the computational version of Theorem~\ref{theorem-TVC-2} for the special case of $S=S^{alg}$ where  $S^{alg}$ is given in terms of algebraic equations.  Although Theorem~\ref{theorem:TVC-alg} is presented in this section as a special case of Theorem~\ref{theorem-TVC-2}, the latter is a generalization of the former from a historical perspective\cite{AG-1973,kyber-81}.
\begin{remark}\label{remark:TVC=+p}
As noted in the statement of Theorem~\ref{theorem:TVC-alg}, \eqref{eq:notTVC} is technically not a transversality condition. Nonetheless, it is apparent that \eqref{eq:notTVC} is integral to the transversality conditions as it helps determine $\bp$.  Furthermore, because the necessary condition for $\bp$ cannot be separated, in general, from the broader transversality condition (see \eqref{eq:TVC-abstract-yes-p} and \eqref{eq:TVC-abstract-yes-p=}) we group it with \eqref{eq:TVC-alg}.
\end{remark}
\begin{remark}\label{rem:KKT}
If the Hamiltonian does not depend upon $\bp$, then $\bsigma = \bzero$; see \eqref{eq:sigma=bydef}.  Then \eqref{eq:notTVC} reduces to the familiar KKT conditions\cite{NW:NumOptBook} for parameter optimization.  As an interesting side note, the connection between transversality conditions and the KKT conditions has been utilized in \cite{rossJCAM-1,rossJCAM-2,ross-CD,ross:JNVA} to construct an optimal control theory for optimization.
\end{remark}

\section{Foundations of Boundary Conditions Defined by Differential Equations}
\label{sec:newBCs}

The transversality conditions that are of main interest in this paper are those conditions when $S$  is given by $S^{ode+}$, an endpoint set defined by differential equations with side conditions.  To better understand what we mean by ``differential equations with side conditions,'' it is instructive to develop a definition of $S^{ode+}$ in stages of increasing intricacies rather than all at once.

\subsection{A Basic Set Parameterized by a Differential Equation }
Consider an initial-value problem (IVP) corresponding to a non-autonomous differential equation given by,
\begin{equation}\label{eq:ivp-a}
\frac{d\bx}{ds} := \bx' = \bg(\bx, s), \quad \bx(s^\sharp) = \bx^\sharp
\end{equation}
where $(\bx^\sharp, s^\sharp) \in \real{N_x} \times \Real $ is a given point in ``space'' and ``time.''  We choose the symbol $s$ for the independent variable instead of $t$ in order to reserve the latter notation for the clock time implied in Section~\ref{sec:reviewTVC}.  The distinctions between $s$ and $t$ will become more apparent in the discussion to follow. Where convenient, we will refer to $s$ as $s$-time or local $s$-time.

If the function $\bg$ satisfies the conditions of the Picard-Lindel\"{o}f theorem\cite{teschl-ode,hale-ode}, then there is a unique solution to \eqref{eq:ivp-a} that passes through the point $(\bx^\sharp, s^\sharp)$; see Fig.~\ref{fig:ODECurve1}.
%
\begin{figure}[h!]
      \centering
      {\parbox{0.9\columnwidth}{
      \centering
      {\includegraphics[width = 0.4\columnwidth]{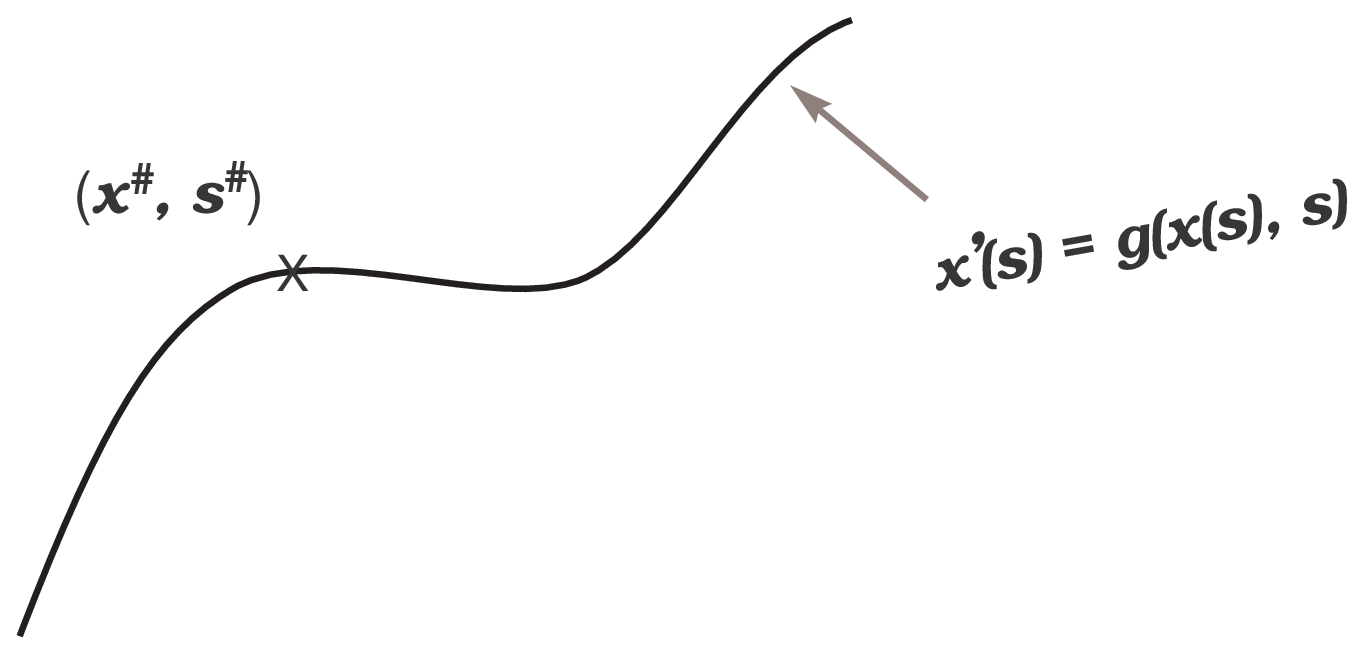}}
      \caption{{Schematic of a ``curve'' corresponding to \eqref{eq:ivp-a}.}}\label{fig:ODECurve1}
      }
      }
\end{figure}
This solution can be described as the set of all points in $\real{N_x} \times \Real$ that satisfies \eqref{eq:ivp-a}.  This set is a ``curve'' parameterized by $s$ and can be written as a ``point-to-set'' map:
\begin{equation}\label{eq:sol-1}
C(\bx^\sharp, s^\sharp) := \set{(\bx(s), s) \in \real{N_x} \times \Real:\ \bx'(s) = \bg(\bx(s), s),  \quad \bx(s^\sharp) = \bx^\sharp}
\end{equation}
The notation $C(\bx^\sharp, s^\sharp)$ implies a point $(\bx^\sharp, s^\sharp)$ must be given (i.e., a side condition) to construct the set defined in the right hand side of \eqref{eq:sol-1}.
%
\begin{remark}\label{rem:picard-infty}
Fig.~\ref{fig:ODECurve1} and \eqref{eq:sol-1} appear to suggest that the solution to \eqref{eq:ivp-a} is valid for infinite $s$-time forwards and backwards; i.e., $s \in (-\infty, \infty)$. This strong assumption is only tentative.  We can easily add bounds to the independent variable, $s$, according to $s^L \le s \le s^U$, where $s^L < s^U \in \Real$.  Furthermore, we may easily incorporate additional rules to $s^L$ and $s^U$ such as dependencies on $(\bx^\sharp, s^\sharp)$ as done in typical textbooks on differential equations\cite{hale-ode,teschl-ode}. In order to construct the ideas piecemeal, we ignore such technicalities in favor of clarity.
\end{remark}
%
From Remark~\ref{rem:picard-infty} it follows that \eqref{eq:sol-1} is not sufficient to define $S^{ode}$. We need one or more side conditions in order to fortify a definition of an endpoint set given by differential equations. 

\subsection{A Manifold Parameterized by Differential Equations }
Now consider another set $P \subset \real{N_x} \times \Real$ from where we pick different initial conditions $(\bx^\sharp, s^\sharp) \in P$.  That is, consider now a new set that generalizes \eqref{eq:sol-1} as follows:
\begin{multline}\label{eq:sol-a}
M := \set{(\bx(s), s) \in C(\bx^\sharp, s^\sharp), (\bx^\sharp, s^\sharp) \in P}\\
  := \set{(\bx(s), s, \bx^\sharp, s^\sharp) \in \left(\real{N_x} \times \Real\right) \times P:\ \bx'(s) = \bg(\bx(s), s),  \quad \bx(s^\sharp) = \bx^\sharp }
\end{multline}
The construction of $M$ is as follows: Pick a point $(\bx^\sharp, s^\sharp) \in P$ and construct $C(\bx^\sharp, s^\sharp)$ according to \eqref{eq:sol-1}.  Pick a different point $(\bx^\sharp, s^\sharp) \in P$ and repeat the process as illustrated in Fig.~\ref{fig:ODECurve2}.
%
\begin{figure}[h!]
      \centering
      {\parbox{0.9\columnwidth}{
      \centering
      {\includegraphics[width = 0.5\columnwidth]{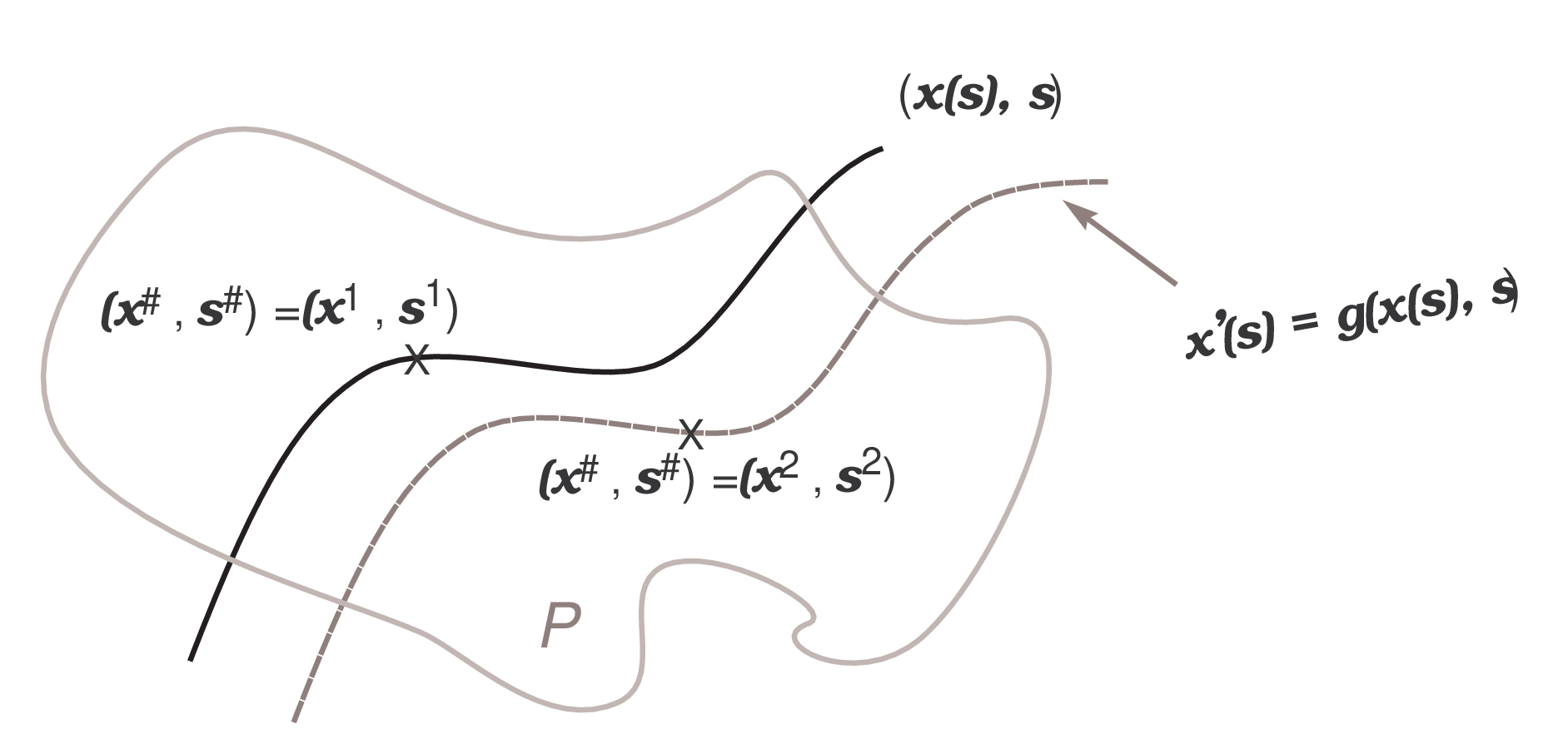}}
      \caption{{Schematic illustrating \eqref{eq:sol-a} as a model for a manifold described by an ODE.}}\label{fig:ODECurve2}
      }
      }
\end{figure}
%
By arguments similar to those of the previous paragraphs, we now have solutions passing through each point in $P$.   Fig.~\ref{fig:ODECurve2} illustrates these ideas as well as a particular rendition of \eqref{eq:sol-a}.  In this illustration, $M \supset C(\bx^\sharp, s^\sharp)$.  Note also that $M$ is not necessarily the same as $P$ as illustrated in Fig.~\ref{fig:ODECurve2}.  Typically, $P$ is contained in $M$.  In \eqref{eq:sol-a}, $P$ is a side condition.  The set of points $(\bx^\sharp, s^\sharp) \in P$ are the initial condition for \eqref{eq:ivp-a}.  These points do not constitute nor limit the candidate initial conditions $(\bx_0, t_0)$ to the set $P$. To sharpen these distinctions, we call $(\bx^\sharp, s^\sharp)$ the phantom initial conditions. 
%
\begin{definition}[phantom initial conditions]\label{def:phantom}
  The set of points $(\bx^\sharp, s^\sharp) \in P$ that generate the flow of the ODE $\bx'(s) = \bg(\bx(s), s)$ are called the phantom initial conditions.
\end{definition}
%

Although $M$ generalizes $C$, it is not yet complete to describe our first goal of defining $S^{ode+}$. To further a definition of $S^{ode+}$, we need a mechanism to connect the local $s$-time with the transfer $t$-time.  To produce this connection, we develop the concept of coordinated and uncoordinated clock times.

\subsection{Coordinated and Uncoordinated Clock Times}

Consider the mathematical problem of defining an initial point $(\bx_0, t_0)$ (of an optimal control problem) in terms of differential equations.  As indicated earlier, the $s$-time in \eqref{eq:ivp-a} and the $t$-time in Section~\ref{sec:reviewTVC} are not necessarily the same. The following example amplifies this point.
\begin{example}\label{ex:1}
Suppose $\bx = (x, y) \in \real{2}$ and we require $\bx_0 = \bx(t_0)$ in an optimal control problem to be selected from a unit circle so that the initial conditions are required to satisfy the constraint, $x^2_0 + y^2_0 = 1$.  See Fig.~\ref{fig:unitCircle}.  This circle can also be parameterized as  $x_0(t) = \sin(t)$ and $y_0(t)=\cos(t)$. Obviously, $t$ in this parametrization has no bearing on the clock time $t$ (or $t_0$).
\end{example}
%
\begin{figure}[h!]
      \centering
      {\parbox{0.9\columnwidth}{
      \centering
      {\includegraphics[width = 0.4\columnwidth]{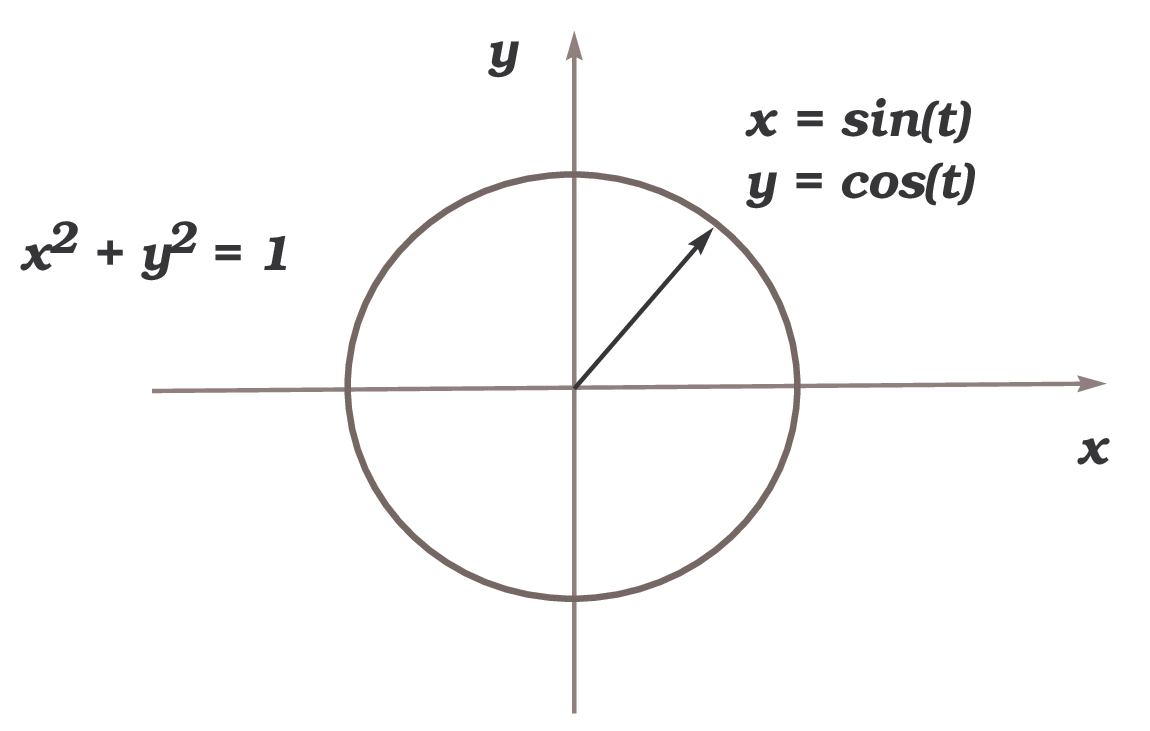}}
      \caption{{Two different parameterizations of a $1$-dimensional manifold.}}\label{fig:unitCircle}
      }
      }
\end{figure}
In the case of Example~\ref{ex:1} when $t$ is not the same as the clock time, the unit circle is parameterized as $(\sin(s), \cos(s))$ instead of $(\sin(t), \cos(t))$.  Using the symbol $s$ instead of $t$ addresses the problem of keeping $s$ and $t$ distinct; however, it requires a few caveats and some additional clarity in notation.  To understand this point, suppose we write,
\begin{equation}\label{eq:x0t0prob-1}
x_0(s) = \sin(s), \quad y_0(s) = \cos(s)
\end{equation}
Then substituting $s= t_0$ in \eqref{eq:x0t0prob-1} generates $x_0(t_0) = \sin(t_0), y_0(t_0) = \cos(t_0)$ which might (incorrectly) be construed to imply that the only valid value of the initial conditions is at $s=t_0$. But note, however, the symbolic difference between $x_0(t_0)$ and $x(t_0)$.  Per \eqref{eq:x0xf-abbrv}, $x(t_0)$ is just $x_0$. The symbol $x_0(t_0)$ represents an allowable value of $x_0$ at $s = t_0$ when it is parameterized as $x_0(s)$.  That is, $x_0(s)$ is a new symbol that represents \emph{all} allowable values of $x_0 = x(t_0)$.   Because $x_0$ itself is parameterized by $s$-time, we use the symbol $(x_0, y_0)$ (resp. $(x_f, y_f)$) without its argument $t_0$ (resp. $t_f$) to rewrite \eqref{eq:x0t0prob-1} as,
\begin{equation}\label{eq:x0t0prob-solved}
x_0 = x_0(s) \Big(:= \sin(s)\Big), \quad y_0 = y_0(s) \Big(:= \cos(s) \Big)
\end{equation}
In other words, the symbol $\bx_0$ implies $\bx(t_0)$ as always but $\bx_0(s)$ implies a parametrization of $\bx_0$ (via the independent variable $s$).  Hence, in \eqref{eq:x0t0prob-solved} $s$ can take any value via the notation $x_0(s)$ and $y_0(s)$; and, no matter the value of $s$, the outputs of $x_0(s)$ and $y_0(s)$, denoted by $x_0$ and $y_0$ respectively, remains on the unit circle.
\begin{remark}\label{rem:x0snotation}
Equation~(\ref{eq:x0t0prob-solved}) suggests that it might be more prudent to use a different symbol (other than $x_0(s), y_0(s)$) for the right-hand-side of the equalities.  Although such notational change might add clarity for this specific purpose, it also has the effect of obscuring the meaning of prior equations (and the ones to follow).  For instance, if the symbol $\bx$ was not used in \eqref{eq:ivp-a} as a means to add extra clarity to \eqref{eq:x0t0prob-solved}, then, the IVP would read,
$$ \frac{d\by}{ds} := \by' = \bg(\by, s), \quad \by(s^\sharp) = \bx^\sharp $$
with the understanding that $\by$ is a proxy for $\bx_0$.
\end{remark}

Returning to the problem at hand, it is also possible to construct optimal control problems wherein $s$ and $t$ are indeed coordinated clock times. Practical examples of such problems are those pertaining to intercept and rendezvous\cite{brysonHo,longuski,ross-book}.  In the case of Example~\ref{ex:1} for instance, it is quite possible that the unit circle may indeed be generated by a physical system going around in a circle. In this situation, the $s$-clock time will certainly matter (mod $2\pi$). To allow for both possibilities, namely, $s$ and $t$ being coordinated or uncoordinated, we define the following:
\begin{definition}[Coordinated and Uncoordinated Clock Times]\label{def:phi} Let $\phi: \Real \times \Real \to \Real$ be a continuously differentiable function of its arguments such that $\phi(s, t) = 0$ has a real-valued unique solution for $s$ in terms of $t$ and vice versa.  If  $\phi \not\equiv 0$, then we say the $s$-time and $t$-time are coordinated (via the function $\phi$).  Alternatively, if $\phi \equiv 0$, then we say the $s$-time and $t$-time are uncoordinated.
\end{definition}
A simple example of  $\phi\not\equiv 0$ is $\phi(s, t) = s-t$.  This particular situation corresponds to the clock time being synchronized.
\begin{definition}[Synchronized Clock Time]\label{def:sync-bc}
The clock times $s$ and $t$ are said to be synchronized if  $\phi(s, t) = s-t$.
\end{definition}

\subsection{A Preliminary Mathematical Model for Endpoint Conditions Governed by Differential Equations}
Using the preceding ideas, we return to the problem of modeling endpoint conditions of an optimal control problem in terms of differential equations. In view of Definition~\ref{def:phi}, we now construct a new set $\widetilde{M}$ based on \eqref{eq:sol-a} as follows:
\begin{equation}\label{eq:pre-S}
\widetilde{M} := \set{(\bx(s), s, \bx^\sharp, s^\sharp) \in M,\ t \in \Real:\   \phi(s, t) = 0}
\end{equation}
Note the following in contrasting $M$ with $\widetilde{M}$:
\begin{enumerate}
\item $s$ and $t$ may be coordinated via $\phi$.  If $\phi \equiv 0$, then $s$ and $t$ are uncoordinated.
\item If $\phi \equiv 0$, then $t$ is all of $\Real$ according to \eqref{eq:pre-S}.  To contain $t$, it is apparent that an additional ``side'' constraint, such as $t \in I \subset \Real$ must be included in \eqref{eq:pre-S}.
\item In view of Remark~\ref{rem:picard-infty}, it is necessary to further constrain $s$ in \eqref{eq:pre-S}.
\end{enumerate}
These points illustrate the fact that additional side conditions must be imposed on $\widetilde{M}$ in order to achieve our first goal of defining $S^{ode+}$.  Because the gap between $\widetilde{M}$ and $S^{ode+}$ is sufficiently small, we use the preceding constructs in the next section to define new boundary conditions for an optimal control problem.

\section{A Formulation of Endpoint Conditions Given by Differential Equations}

Taking all of the considerations of Section~\ref{sec:newBCs} into account and using \eqref{eq:pre-S} as a starting point, we now define initial and final sets for the endpoint constraints of an optimal control problem in their primitive, constituent form in the following manner:
\begin{subequations}\label{eq:S0+Sf-def}
\begin{multline}\label{eq:S0=def}
S_0^{ode+} := \Big\{(\bx_0, t_0, s_a, \bx^\sharp_a, s^\sharp_a):\ \bx_0 = \bx_0(s_a), \quad \bx_0'(s_a) = \bg_0(\bx_0(s_a), s_a), \quad \bx_0(s^\sharp_a) = \bx^\sharp_a, \\
(\bx^\sharp_a, s^\sharp_a) \in P_a, \quad
\phi_0(s_a,t_0)=0,  \quad t_0 \in I_0, \quad s_a \in Q_a \Big\}
\end{multline}
\begin{multline}\label{eq:Sf=def}
S_f^{ode+} := \Big\{(\bx_f, t_f, s_b, \bx^\sharp_b, s^\sharp_b):\  \bx_f = \bx_f(s_b), \quad \bx_f'(s_b) = \bg_f(\bx_f(s_b), s_b), \quad \bx_f(s^\sharp_b) = \bx^\sharp_b, \\
(\bx^\sharp_b, s^\sharp_b) \in P_b, \quad
\phi_f(s_b, t_f)=0,  \quad t_f \in I_f, \quad s_b \in Q_b \Big\}
\end{multline}
\end{subequations}
In \eqref{eq:S0+Sf-def}, $I_0 \subset \Real$ and $I_f \subset \Real$ are intervals of allowable clock times for $t_0$ and $t_f$ respectively.  Similarly, $Q_a \subset \Real$ and $Q_b \subset \Real$ are intervals of allowable local $s$-times for $s_a$ and $s_b$ respectively.  A schematic for visualizing these sets is shown in Fig.~\ref{fig:ODECurve3}.
%
\begin{figure}[h!]
      \centering
      {\parbox{0.95\columnwidth}{
      \centering
\subfigure[Illustration of some symbols used to define $S_0^{ode+}$ in \eqref{eq:S0=def} ]{
      {\includegraphics[width = 0.43\columnwidth]{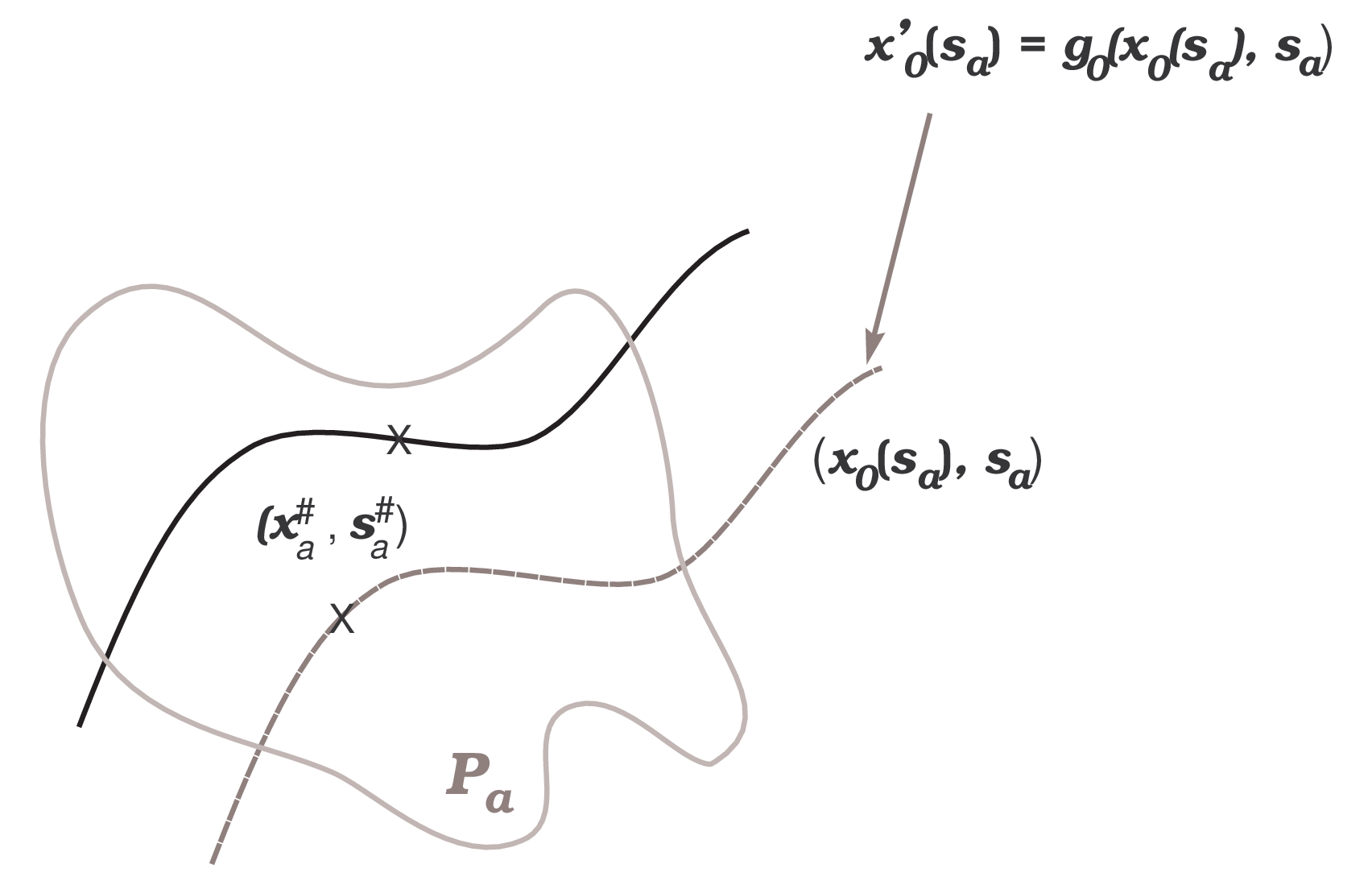}}
}
\hfill
\subfigure[Illustration of some symbols used to define $S_f^{ode+}$ in \eqref{eq:Sf=def}]{
      {\includegraphics[width = 0.43\columnwidth]{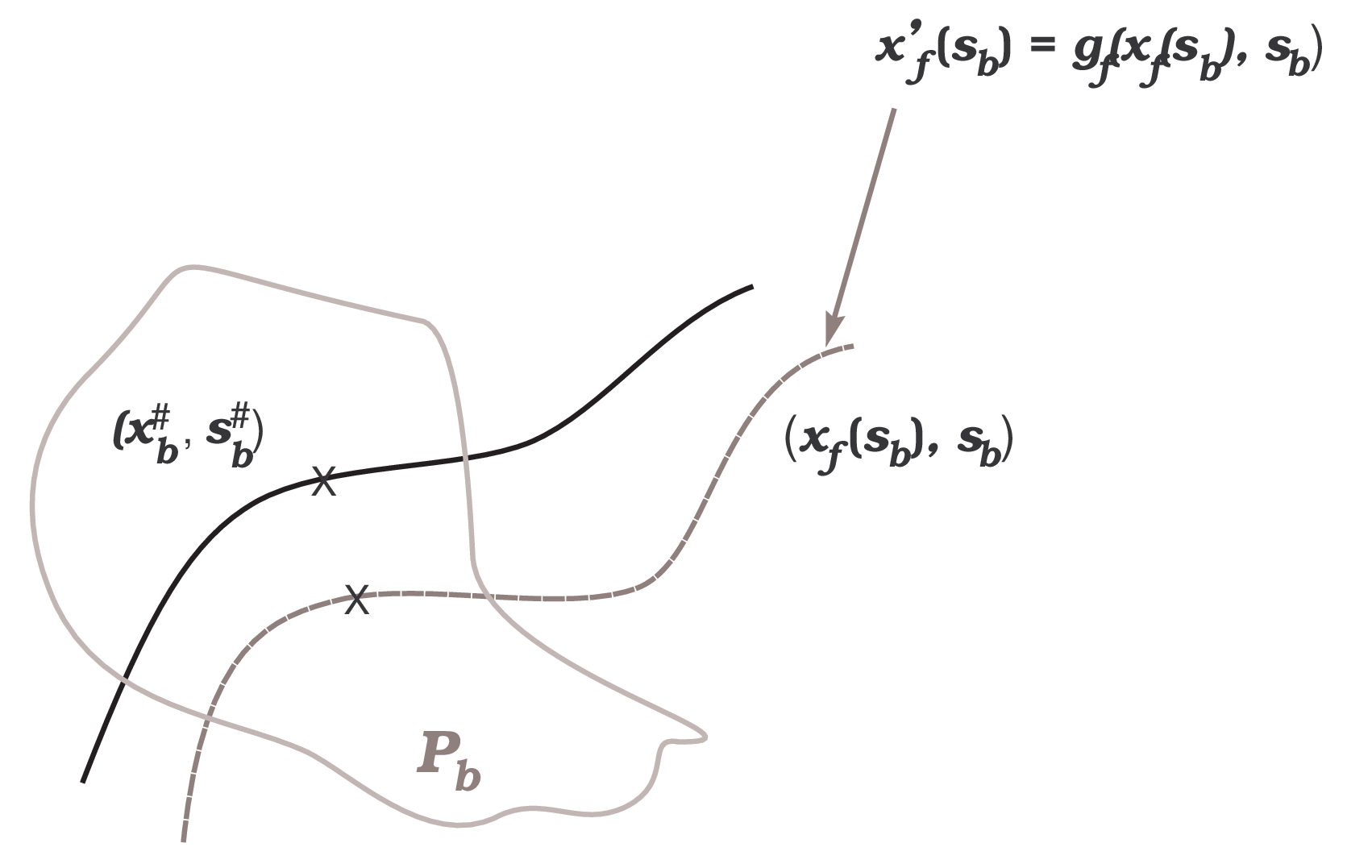}}
}
      \caption{{A schematic for visualizing the sets $S_0^{ode+}$ and $S_f^{ode+}$ defined in \eqref{eq:S0=def} and \eqref{eq:Sf=def} respectively.}}\label{fig:ODECurve3}
      }
      }
\end{figure}
%

Additional explanations with regards to \eqref{eq:S0+Sf-def} are in order:
\begin{enumerate}
\item Dropping subscripts $a$ and $b$ in \eqref{eq:S0=def} and \eqref{eq:Sf=def} respectively, we get the following fundamental set that extends $\widetilde{M}$ defined in \eqref{eq:pre-S}:
\begin{multline}\label{eq:Sg=}
S_g^{ode+} := \Big\{(\bx, t, s, \bx^\sharp, s^\sharp):\ \bx = \bx(s), \quad \bx'(s) = \bg(\bx(s), s), \quad \bx(s^\sharp) = \bx^\sharp, \\
(\bx^\sharp, s^\sharp) \in P, \quad
\phi(s, t)=0,  \quad t \in I, \quad s \in Q \Big\}
\end{multline}
\item Comparing \eqref{eq:Sg=} with \eqref{eq:pre-S} it follows that the former differs from the latter in two ways:
    \begin{enumerate}
        \item An introduction of the equation $\bx = \bx(s)$ to denote the difference between the variable $\bx$ and its parametrization given by $\bx(s)$ along the lines of \eqref{eq:x0t0prob-solved}.  This equation links the remainder of the conditions that defines $S_g^{ode+}$.
        \item The inclusion of two constraints for the clock times $t$ and $s$ in terms of the interval sets $I$ and $Q$ respectively.
    \end{enumerate}
In other words, \eqref{eq:Sg=} (and hence, \eqref{eq:S0=def} and \eqref{eq:Sf=def}) is simply a culmination of the process that began with \eqref{eq:ivp-a}.
\item The centerpiece of \eqref{eq:Sg=} is the differential equation $\bx'(s) = \bg(\bx(s), s)$. In Section~\ref{sec:Intro}, there was an implication that this differential equation was generated from $\dot\bx = \bff(\bx, \bzero, t)= \bg(\bx, t)$ with $t$ replaced by $s$.  That is, the vector field $\bg$ is simply the result of no control action in $\bff$.  Note, however, that $\bg$ in \eqref{eq:Sg=} (and hence $\bg_0$ in \eqref{eq:S0=def} and $\bg_f$ in \eqref{eq:Sf=def}) is more general than no control action.  In fact, the initial set may be defined by an open- or closed-loop control action $\bu = \bk(\bx, t)$, in which case we have $\bg(\bx, t) := \bff(\bx, \bk(\bx, t), t)$.  See, for example, \cite{artificalHalo-2022} and \cite{LTorbits-2020} for practical examples of ``artificial'' manifolds. Thus, if the initial and final sets were governed by two different control strategies, say, $\bk_0(\bx, t)$ and $\bk_f(\bx, t)$, then, we would indeed have two different $\bg$-functions, $\bg_0$ and $\bg_f$ corresponding to $\bk_0$ and $\bk_f$ respectively.  Equations~(\ref{eq:S0=def}) and (\ref{eq:Sf=def}) allows for this practical possibility.
\item The arguments of $\bg_0$ and $\bg_f$ are different in \eqref{eq:S0=def} and \eqref{eq:Sf=def} because these variables apply only to $\bx_0$ and $\bx_f$ respectively. For the same reason, the $s$-time variables are also different and denoted by $s_a$ and $s_b$.
\item The $\phi$ functions are also different in the definitions of $S_0$ and $S_f$ to allow for the possibility that the initial set might just be a manifold where $s_a$ is not coordinated with $t_0$ but the final set might correspond to a rendezvous where $s_b$ and $t_f$ would need to be synchronized.
\item In principle, the constraint $s \in Q$ in \eqref{eq:Sg=} (and hence the constraints $s_a \in Q_a$ and $s_b \in Q_b$ in \eqref{eq:S0=def} and \eqref{eq:Sf=def} respectively) addresses the technicality associated with the validity of the Picard-Lindel\"{o}f theorem.  See Remark~\ref{rem:picard-infty}.  Nonetheless, this constraint is quite practical in situations such as the one shown in Fig.~\ref{fig:unitCircle} where it is sufficient for $s$ to be in the interval $[0, 2\pi]$.
\item Because $s^\sharp$ is a specific value of $s$, it is implicitly required to satisfy the constraint $s^\sharp \in Q$.  Hence, an explicit specification of this constraint is excluded in the definition of $S_g^{ode+}$.
\item The constraint $t \in I$ in \eqref{eq:Sg=} (and hence, $t_0 \in I_0$ and $t_f \in I_f$ in \eqref{eq:S0=def} and \eqref{eq:Sf=def} respectively) is included for practical operational reasons.  Consider, for example, $t_0 \in I_0$.  The initial time $t_0$ in a real-world space operation may be constrained for a myriad of reasons such as the availability of communication links to perform an operation, completion of a prior operation and so on.
\item In \eqref{eq:S0=def} and \eqref{eq:Sf=def}, we implicitly assume the time intervals $I_0$ and $I_f$ are chosen by the modeler/operator in such a way that there are no obvious conflicts in clock times, such as, for example, requiring $t_0$ to be greater than $t_f$.  Let $I_0$ and $I_f$ be defined by,
    \begin{equation}\label{eq:I0Ifdef}
    I_0:= \set{t_0 \in \Real:\ t_0^L \le t_0 \le t_0^U} \quad  I_f:= \set{t_f \in \Real:\ t_f^L \le t_f \le t_f^U}
    \end{equation}
    where $(t_0^L, t_f^L)$ and $(t_0^U, t_f^U)$ denote the lower and upper bounds on the initial and final clock times respectively. We implicitly assume $t_0^U \ge t_0^L$ and $t_f^U \ge t_f^L$.  An equality in the prior sentence corresponds to a choice of fixed clock times.  Thus, a ``sufficient'' practical condition for having no conflicts in clock times is for the modeler/operator to choose $t_f^L > t_0^U$.  Or alternatively $I_0 \cap I_f = \emptyset$.
\end{enumerate}
%

\begin{remark}\label{rem:picard-cara}
Many guidance and control systems involve discontinuous control functions $t \mapsto \bu(t)$\cite{ross-book}. Consequently, if the function $\bg_0$ (resp. $\bg_f$) in \eqref{eq:S0=def} (resp. \eqref{eq:Sf=def}) is generated by a discontinuous control function, then  $\bg_0$ (resp. $\bg_f$) will be discontinuous with respect to the independent variable $s_a$ (resp. $s_b$). In this case, $\bg_0$ (resp. $\bg_f$) will not satisfy the conditions of the Picard-Lindel\"{o}f theorem.  However, the process of generating \eqref{eq:S0=def} (resp. \eqref{eq:Sf=def}) starting from \eqref{eq:ivp-a} still holds with the caveat that $\bg_0$ (resp. $\bg_f$) must now satisfy the Carath\'{e}odory conditions\cite{hale-ode}.
\end{remark}

From the preceding discussions, it follows that the triple, $(s, \bx^\sharp, s^\sharp)$, are parameters that must be defined and properly constrained in order to construct $S^{ode+}$.  This point explains why Theorem~\ref{theorem-TVC-2} will form the basis of the development of new transversality conditions.  Before proceeding to develop the new transversality condition, we formalize the definition of $S^{ode+}$.

Let $\bp_a:= (s_a, \bx_a^\sharp, s_a^\sharp)$ and $\bp_b:= (s_b, \bx_b^\sharp, s_b^\sharp)$ denote the parameters defined in \eqref{eq:S0=def} and \eqref{eq:Sf=def} respectively.  Then, according to \eqref{eq:S0+Sf-def}, the constraints on $\bp_a$ and $\bp_b$ are given by the Cartesian products,
\begin{subequations}\label{eq:pa+pb-constr}
  \begin{align}
  \bp_a &:=  (s_a, \bx^\sharp_a, s^\sharp_a) \in (Q_a \times P_a) \\
  \bp_b &:=  (s_b, \bx^\sharp_b, s^\sharp_b) \in (Q_b \times P_b) 
  \end{align}
\end{subequations} 
Define,
\begin{equation}\label{eq:Sode+=bydef}
  S^{ode+} :=  S_0^{ode+} \times  S_f^{ode+}
\end{equation}
where $S_0^{ode+}$ and  $S_f^{ode+}$ are given by \eqref{eq:S0=def} and \eqref{eq:Sf=def} respectively. Then the boundary conditions defined by differential equations with side conditions are given by,
\begin{equation}\label{eq:bc=ode+}
  (\bx_0, t_0, \bp_a; \bx_f, t_f, \bp_b) \in S^{ode+}
\end{equation}

\section{The New Transversality Conditions}\label{sec:newTVCs}

To formally state the new transversality conditions, we need the concept of a weak adjoint covector.  This idea is best understood by contrasting it with the usual (i.e., strong) definition of an adjoint covector associated with a differential equation $\bx'(\xi) = \bg(\bx(\xi), \xi)$ defined over $\xi \in [0, \tau]$. To support the definition of a weak adjoint covector, we use the symbol $\mathbf{1}_{[0, s]}(\xi)$ to denote an indicator function defined by,
\begin{equation}\label{eq:indicatorDef}
  \mathbf{1}_{[0, s]}(\xi) := \left\{
                                           \begin{array}{ll}
                                             \mathbf{1} & \text{if }  \xi \in [0, s] \\
                                             \bzero & \text{otherwise}
                                           \end{array}
                                         \right.
\end{equation}
%
\begin{definition}[strong adjoint covector]
 A strong adjoint covector is a function $\xi \mapsto \blam(\xi)$ in dual space that satisfies the pointwise (strong) constraint functional equation,
$$ \int_{0}^{\tau} \blam^T(\xi) \big(\bx'(\xi) - \bg(\bx(\xi), \xi) \big)\, d\xi  = 0$$
for all trajectories $\xi \mapsto \bx(\xi)$ that satisfy the differential constraint  $\bx'(\xi) - \bg(\bx(\xi), \xi) = \bzero$. 
\end{definition}
\begin{definition}[weak adjoint covector]\label{def:weakAdj}
Let $\odot$ denote a Hadamard product and $\bpsi$ a constant covector.  A weak adjoint covector is a family of dual space functions $\xi \mapsto \widehat{\blam}(\xi) := {\bpsi}\odot \mathbf{1}_{[0, s]}(\xi)$ parameterized by $\bpsi$ and indexed by $s$ that satisfy the weak constraint functional equation,
\begin{equation}\label{eq:weakAdj-def}
   \int_{0}^{s} \widehat{\blam}^T(\xi) \big(\bx'(\xi) - \bg(\bx(\xi), \xi) \big)\, d\xi  = 0 \quad \forall\ s \in (0, \tau]
\end{equation}
for all trajectories $\xi \mapsto \bx(\xi)$ that satisfy the differential constraint  $\bx'(\xi) - \bg(\bx(\xi), \xi) = \bzero$ over almost all $\xi \in [0, \tau]$. 
\end{definition}
%
Note that our use of the word weak does not imply the differential equations are imposed loosely.  The test of \eqref{eq:weakAdj-def} is performed over all windows in $(0, \tau]$ by the requirement $\forall\ s \in (0, \tau]$.  A single test is weak but an infinite family of weak tests over every possible window enforces the same result as the strong form but achieves it using an alternative mathematical framework. In the rest of this paper, it will be convenient to use the term weak adjoint covector to denote the constant vector $\bpsi$ instead of $\widehat{\blam}(\xi) := {\bpsi}\odot \mathbf{1}_{[0, s]}(\xi)$.

In preparing to formulate the new transversality conditions, we first note that the parameters $\bp_a$ and $\bp_b$ are exclusive to the boundary conditions.  Hence, they do not appear in the dynamics function, $\bff$. As a result, the vector $\bsigma$ defined in \eqref{eq:sigma=bydef} and denoted here as $\bsigma_{ab}$ simplifies to,
\begin{equation}\label{eq:sigma=bydef-2}
\bsigma_{ab}:=  
\left[
    \begin{array}{c}
      \sigma_{s_a}\\
      \bsigma_{\bx^\sharp_a} \\
      \sigma_{s^\sharp_a} \\
      \sigma_{s_b} \\
      \bsigma_{\bx^\sharp_b} \\
      \sigma_{s^\sharp_b}\\
    \end{array}
  \right] :=
\left[
\begin{array}{c}
\nu^0\,\displaystyle\int_{t_0}^{t_f} \partial_{s_a} F(\bx(t), \bu(t), t; s_a, \bx^\sharp_a, s^\sharp_a, s_b, \bx^\sharp_b, s^\sharp_b)\, dt \\[1.5em]
\nu^0\,\displaystyle\int_{t_0}^{t_f} \partial_{\bx^\sharp_0} F(\bx(t), \bu(t), t; s_a, \bx^\sharp_a, s^\sharp_a, s_b, \bx^\sharp_b, s^\sharp_b)\, dt \\[1.5em]
\nu^0\,\displaystyle\int_{t_0}^{t_f} \partial_{s_a^\sharp} F(\bx(t), \bu(t), t; s_a, \bx^\sharp_a, s^\sharp_a, s_b, \bx^\sharp_b, s^\sharp_b)\, dt \\[1.5em]
\nu^0\,\displaystyle\int_{t_0}^{t_f} \partial_{s_b} F(\bx(t), \bu(t), t; s_a, \bx^\sharp_a, s^\sharp_a, s_b, \bx^\sharp_b, s^\sharp_b)\, dt \\[1.5em]
\nu^0\,\displaystyle\int_{t_0}^{t_f} \partial_{\bx^\sharp_b} F(\bx(t), \bu(t), t; s_a, \bx^\sharp_a, s^\sharp_a, s_b, \bx^\sharp_b, s^\sharp_b)\, dt \\[1.5em]
\nu^0\,\displaystyle\int_{t_0}^{t_f} \partial_{s^\sharp_b} F(\bx(t), \bu(t), t; s_a, \bx^\sharp_a, s^\sharp_a, s_b, \bx^\sharp_b, s^\sharp_b)\, dt \\
                                                                                     \end{array}
                                                                                   \right]
\end{equation}

We now state the new transversality conditions in the form of the following theorem.  

\begin{theorem}[Main Result]\label{theorem:ross}
Assume the following:
\begin{enumerate}
\item Problem~$(\bP^{ode+})$ is defined by Problem~$(\bP)$ with $S$ given by $S^{ode+}$ (Cf.~\eqref{eq:probPbold} and \eqref{eq:Sode+=bydef}) and boundary conditions given by \eqref{eq:bc=ode+}.
\item The endpoint- and running cost functions for Problem~$(\bP^{ode+})$ are given by functions,
\begin{subequations}
\begin{align}
E(\bx_0, t_0, \bp_a; \bx_f, t_f, \bp_{b}) &:= E(\bx_0, t_0, s_a, \bx^\sharp_a, s^\sharp_a; \bx_f, t_f,  s_b, \bx^\sharp_b, s^\sharp_b)\\
F(\bx, \bu, t, \bp_a, \bp_{b}) &:=F(\bx, \bu, t; s_a, \bx^\sharp_a, s^\sharp_a, s_b, \bx^\sharp_b, s^\sharp_b)
\end{align}
\end{subequations}
that are continuously differentiable with respect to their arguments.
\item The vectors in the normal cones are given by,
\begin{subequations}
\begin{align}
\left(\nu_{t_0}, \nu_{t_f} \right) & \in N_{I_0}(t_0) \times N_{I_f}(t_f) \label{eq:nut0inN}\\
\left(\eta_{s_a}, \eta_{s_b}\right) & \in N_{Q_a}(s_a) \times  N_{Q_b}(s_b) \label{eq:eta_sasb}\\
\left(\etab_{x^\sharp_a}, \eta_{s^\sharp_a} \right) & \in N_{P_a}(\bx^\sharp_a, s^\sharp_a)\\
\left(\etab_{x^\sharp_b}, \eta_{s^\sharp_b} \right) & \in N_{P_b}(\bx^\sharp_b, s^\sharp_b)
\end{align}
\end{subequations}
\item The clock-coordination multipliers associated with enforcing the constraints $\phi_0(s_a, t_0) = 0$ and $\phi_f(s_b, t_f) = 0$ are given by $\nu^C_a$ and $\nu^C_b$  respectively.
\item The weak adjoint covectors (Cf.~Definition~\ref{def:weakAdj}) associated with enforcing the boundary differential equations, $\bx'_0(s_a) = \bg_0(\bx_0(s_a), s_a) $ and $\bx'_f(s_b) = \bg_f(\bx_f(s_b), s_b) $  are given by $\bpsi_0 \in \real{N_x}$ and $\bpsi_f \in \real{N_x}$ respectively.
\end{enumerate}

Then, the transversality conditions for Problem~$(\bP^{ode+})$ are given by,
\begin{subequations}\label{eq:ross-tvc-lam0+f}
\begin{align}
- \bpsi_0 & =  \blam(t_0) + \nu^0\,\partial_{\bx_0} E(\bx_0, t_0, \bp_a; \bx_f, t_f, \bp_b) \label{eq:nu0-def}\\
 \bpsi_f & = \blam(t_f) - \nu^0\, \partial_{\bx_f}E(\bx_0, t_0, \bp_a; \bx_f, t_f, \bp_b) \label{eq:nuf-def}\\
\bpsi_0^T \bg_0(\bx_0, s_a) &=
 \nu^0\, \partial_{s_a}E(\bx_0, t_0, s_a, \bx^\sharp_a, s^\sharp_a; \bx_f, t_f,  s_b, \bx^\sharp_b, s^\sharp_b)
 + \nu_a^C \partial_{s_a}\phi_0(s_a, t_0) + \eta_{s_a}+\sigma_{s_a} \label{eq:ross-tvc-lam0}\\
\bpsi_f^T \bg_f(\bx_f, s_b) &=
\nu^0\, \partial_{s_b}E(\bx_0, t_0, s_a, \bx^\sharp_a, s^\sharp_a; \bx_f, t_f,  s_b, \bx^\sharp_b, s^\sharp_b)
 + \nu_b^C \partial_{s_b}\phi_f(s_b, t_f) + \eta_{s_b} + \sigma_{s_b} \label{eq:ross-tvc-lamf}
\end{align}
\end{subequations}
where $\sigma_{s_a}$ and $\sigma_{s_b}$ are as defined in \eqref{eq:sigma=bydef-2}.
The Hamiltonian value conditions are given by,
\begin{subequations}\label{eq:ross-hvc}
\begin{align}
\mathcal{H}[@t_0] - \nu^0\, \partial_{t_0}E(\bx_0, t_0, s_a, \bx^\sharp_a, s^\sharp_a; \bx_f, t_f,  s_b, \bx^\sharp_b, s^\sharp_b) - \nu_{t_0} &= \nu_a^C \partial_{t_0}\phi_0(s_a, t_0) \label{eq:ross-hvc-0}\\
\mathcal{H}[@t_f] + \nu^0\,\partial_{t_f}E(\bx_0, t_0, s_a, \bx^\sharp_a, s^\sharp_a; \bx_f, t_f,  s_b, \bx^\sharp_b, s^\sharp_b) + \nu_{t_f} & = - \nu_b^C \partial_{t_f}\phi_f(s_b, t_f) \label{eq:ross-hvc-f}
\end{align}
\end{subequations}
The optimality conditions for the remainder of the constituent parameters namely, $\bx^\sharp_a$, $s^\sharp_a$,  $\bx^\sharp_b$ and $s^\sharp_b$ are given by,
\begin{subequations}\label{eq:ross-param}
\begin{align}
 \nu^0\, \partial_{\bx^\sharp_a}E(\bx_0, t_0, s_a, \bx^\sharp_a, s^\sharp_a; \bx_f, t_f,  s_b, \bx^\sharp_b, s^\sharp_b) + \bsigma_{x^\sharp_a} &=
\bpsi_0 - \etab_{x^\sharp_a} \label{eq:ross-param-a1}\\
 \nu^0\, \partial_{s^\sharp_a}E(\bx_0, t_0, s_a, \bx^\sharp_a, s^\sharp_a; \bx_f, t_f,  s_b, \bx^\sharp_b, s^\sharp_b) +\sigma_{s^\sharp_a} &=
- \left(\bpsi_0^T\bg_0(\bx^\sharp_a, s^\sharp_a) + \eta_{s^\sharp_a} \right) \label{eq:ross-param-a2}\\
\nu^0\, \partial_{\bx^\sharp_b}E(\bx_0, t_0, s_a, \bx^\sharp_a, s^\sharp_a; \bx_f, t_f,  s_b, \bx^\sharp_b, s^\sharp_b) + \bsigma_{x^\sharp_b} &=
\bpsi_f - \etab_{x^\sharp_b}\\
\nu^0\, \partial_{s^\sharp_b}E(\bx_0, t_0, s_a, \bx^\sharp_a, s^\sharp_a; \bx_f, t_f,  s_b, \bx^\sharp_b, s^\sharp_b) + \sigma_{s^\sharp_b} &=
-\left(
\bpsi_f^T\bg_f(\bx^\sharp_b, s^\sharp_b)  + \eta_{s^\sharp_b} \right)
\end{align}
\end{subequations}
where $\bsigma_{x^\sharp_a}$, $\sigma_{s^\sharp_a}$, $\bsigma_{x^\sharp_b}$ and $\sigma_{s^\sharp_b}$ are as defined in \eqref{eq:sigma=bydef-2}.

\end{theorem}

We defer a proof of Theorem~\ref{theorem:ross} to Section~\ref{sec:proof} in favor of discussing the result first and its many corollaries.

\section{Discussion of the Main Result and Development of Some Corollaries}
Theorem~\ref{theorem:ross} contains a number of new results.  Before discussing these results, we first note the following points about the theorem itself:
\begin{enumerate}
\item The transversality conditions are coupled to the optimality conditions associated with the parameters $\bp_{a}$ and $\bp_b$ (defined in \eqref{eq:pa+pb-constr}). This coupling was ``predicted'' in Remark~\ref{remark:TVC=+p}.  In other words, the standard result of Theorem~\ref{theorem-TVC-1} was not sufficiently general to support Theorem~\ref{theorem:ross}.
\item Because the weak adjoint covectors are given explicitly by Eqs.~(\ref{eq:nu0-def}) and (\ref{eq:nuf-def}), they can be eliminated in the remainder of Theorem~\ref{theorem:ross} by substituting these expressions for $\bpsi_0$ and $\bpsi_f$ respectively. However, they can also be used on their own as part of an independent process for verifying and validating the extremality of a computed solution.  See \cite{dixon-diss-2025} for examples.
\item Equations~(\ref{eq:ross-tvc-lam0}) and (\ref{eq:ross-tvc-lamf}) are two new scalar equations that help determine the initial and final transversality conditions associated with the costates.  The reason these equations are not $N_x$-dimensional each is because the initial and final states are each required to satisfy $N_x$-dimensional differential equations.  Hence, once an ``initial point,'' $(\bx^\sharp, s^\sharp)$, for an $N_x$-dimensional differential equation (Cf.~\eqref{eq:ivp-a}) is determined, the only free parameter is the local time-like variable $s$.  Thus, \eqref{eq:ross-tvc-lam0} and \eqref{eq:ross-tvc-lamf} are two scalar transversality equations providing the ``missing'' or ``natural'' boundary conditions\cite{brysonHo,longuski} for two scalars $s_a$ and $s_b$.
\item The statement $\nu_{t_0} \in N_{I_0}(t_0)$ in \eqref{eq:nut0inN} (and correspondingly $\nu_{t_f} \in N_{I_f}(t_f)$) reduces to the familiar complementarity conditions:
\begin{equation}\label{eq:normal-cone-1D}
\nu_{t_0} \in N_{I_0}(t_0) \Leftrightarrow \nu_{t_0} \left\{
                                                 \begin{array}{ll}
                                                   \le 0, & \hbox{\text{if} }\  t_0 = t_0^L \\
                                                   = 0, & \hbox{if }\  t_0^L < t_0 < t_0^U \\
                                                   \ge 0, & \hbox{if }\  t_0 = t_0^U \\
                                                   \text{any value}, & \hbox{if }\  t^L_0 = t_0^U
                                                 \end{array}
                                               \right.
\end{equation}
Equation~(\ref{eq:normal-cone-1D}) follows from \eqref{eq:eta-meaning}. Equation~(\ref{eq:normal-cone-1D}) can also be ``derived'' by observation of Fig.~\ref{fig:boxNC}.
\end{enumerate}

The transversality conditions given by Theorem~\ref{theorem:ross} incorporate many number of practical considerations such as clock-time synchronization.  To better understand this theorem, it is useful to break it down to simpler situations so that different aspects of the theorem can be explained in some detail. This is done in terms of several corollaries to follow. We begin with a deceptively simple case, namely, when the differential equations (that govern the boundary conditions) vanish.

\subsection{Case 1: Vanishing Boundary Differential Equations}
Consider the case when the boundary differential equations vanish, namely,
\begin{equation}\label{eq:ode=0}
\bx'_0(s_a) = \bg_0(\bx_0(s_a), s_a) = \bzero \quad\text{and} \quad \bx'_f(s_b) = \bg_f(\bx_f(s_b), s_b) = \bzero
\end{equation}
At first sight, \eqref{eq:ode=0} suggests that Theorem~\ref{theorem:ross} should reduce to the case corresponding to fixed boundary conditions such as the situation depicted in panel~$(a)$ in Fig.~\ref{fig:Examples4S}.  In fact, this is not necessarily true; rather, \eqref{eq:ode=0} implies the case depicted in panel~$(c)$ (of Fig.~\ref{fig:Examples4S}). To prove this point, we analyze the implication of \eqref{eq:ode=0} on only the initial-time conditions.  Results for the final-time conditions follow similarly.

The condition $\bx'_0(s_a) = \bzero$ implies $\bx_0(s_a)= \bx^\sharp_a $ for all $ s_a \in Q_a$.  See \eqref{eq:S0=def}.  Hence, we get,
\begin{equation}\label{eq:x0=xSharp}
\bx_0 = \bx_0(s_a) = \bx^\sharp_a \quad \big(= \bx_0(s^\sharp_a)\big)
\end{equation}
Note, however, that $\bx^\sharp_a$ is not necessarily fixed; rather, we have $(\bx^\sharp_a, s^\sharp_a) \in P_a$.  Hence, \eqref{eq:x0=xSharp} implies
\begin{equation}\label{eq:ode=0-firstResult}
(\bx_0, s^\sharp_a) \in P_a
\end{equation}
Thus \eqref{eq:ode=0-firstResult} is suggestive of panel~$(c)$ in Fig.~\ref{fig:Examples4S} rather than panel~$(a)$, unless, of course if $P_a$ is a singleton.

For brevity of discussion, we now assume $Q_a = \Real$.  Thus, the initial set corresponding to the case of a vanishing boundary ODE is given by,
\begin{equation}\label{eq:S04g=0}
S_0 = \Big\{(\bx_0, t_0) \in \real{N_x} \times \Real, s_a \in \Real, (\bx^\sharp_a, s^\sharp_a) \in P_a:\ \bx_0 = \bx^\sharp_a, \quad
\phi_0(s_a,t_0)=0,  \quad t_0 \in I_0, \quad \Big\}
\end{equation}
The meaning of a vanishing boundary ODE can now be interpreted in terms of the variables of \eqref{eq:S04g=0} as follows:
\begin{enumerate}
\item $\bx_0$ is the same as the optimization parameter $\bx^\sharp_a$.  Because $\bx^\sharp_a$ is not necessarily a singleton, neither is $\bx_0$.
\item $t_0$ is not necessarily a singleton.
\item The variable $s^\sharp_a$ is effectively a dummy variable that does not directly contribute to constraining $\bx_0$ or $t_0$.
\end{enumerate}
For simplicity of discussion, we now assume the running cost $F$ to be independent of the parameterd $\bp_{a}$ and $\bp_b$ (see \eqref{eq:pa+pb-constr}).  This assumption merely simplifies the discussions to follow by setting $\bsigma_{ab} = \bzero$ (see \eqref{eq:sigma=bydef-2}).   Thus, for example, \eqref{eq:ross-param-a2} reduces to,
\begin{equation}\label{eq:KKT4saSharp}
 \nu^0\, \partial_{s^\sharp_a}E(\bx_0, t_0, s_a, \bx^\sharp_a, s^\sharp_a; \bx_f, t_f,  s_b, \bx^\sharp_b, s^\sharp_b) =
-   \eta_{s^\sharp_a}
\end{equation}
Equation~(\ref{eq:KKT4saSharp}) is just the generalized KKT condition for optimality of $s^\sharp_a$. See Theorem~\ref{theorem-TVC-2} and  Remark~\ref{rem:KKT} for context.  Per the discussion point earlier, \eqref{eq:KKT4saSharp} was ``expected,'' because $s^\sharp_a$ does not contribute to constraining $\bx_0$ or $t_0$.

Next, consider \eqref{eq:ross-tvc-lam0}.  Because $s_a \in \Real$ per \eqref{eq:S04g=0}, $\eta_{s_a} = 0$.  Hence, \eqref{eq:ross-tvc-lam0} simplifies to,
\begin{equation}\label{eq:sa4g=0}
 \nu^0\, \partial_{s_a}E(\bx_0, t_0, s_a, \bx^\sharp_a, s^\sharp_a; \bx_f, t_f,  s_b, \bx^\sharp_b, s^\sharp_b)
 + \nu_a^C \partial_{s_a}\phi_0(s_a, t_0) = 0
\end{equation}
Define,
\begin{equation}\label{eq:Ebar4g=0}
\overline{E}(\nu_a^C; \bx_0, t_0, s_a, \bx^\sharp_a, s^\sharp_a; \bx_f, t_f,  s_b, \bx^\sharp_b, s^\sharp_b) :=     \nu^0\, E(\bx_0, t_0, s_a, \bx^\sharp_a, s^\sharp_a; \bx_f, t_f,  s_b, \bx^\sharp_b, s^\sharp_b)
 + \nu_a^C \phi_0(s_a, t_0)
\end{equation}
Then \eqref{eq:sa4g=0} can be rewritten as,
\begin{equation}\label{eq:sa4g=0-Lag}
\partial_{s_a}\overline{E}(\nu_a^C; \bx_0, t_0, s_a, \bx^\sharp_a, s^\sharp_a; \bx_f, t_f,  s_b, \bx^\sharp_b, s^\sharp_b) = 0
\end{equation}
In other words, \eqref{eq:sa4g=0} is a first-order necessary condition for minimizing $E$ with respect to $s_a$ under the constraint $\phi_0(s_a, t_0) = 0$ with \eqref{eq:Ebar4g=0} as the Lagrangian for this problem.  Per Remark~\ref{rem:KKT}, \eqref{eq:sa4g=0-Lag} and hence \eqref{eq:sa4g=0} is an expected result since $s_a$ is an optimization parameter.  Similarly, \eqref{eq:ross-hvc-0} may be written as,
\begin{equation}\label{eq:hvc-0-4sync}
\mathcal{H}[@t_0] =  \partial_{t_0}\overline{E}(\nu_a^C; \bx_0, t_0, s_a, \bx^\sharp_a, s^\sharp_a; \bx_f, t_f,  s_b, \bx^\sharp_b, s^\sharp_b) + \nu_{t_0}
\end{equation}
which is the standard Hamiltonian value condition\cite{ross-book,vinter} at $t_0$.  To generate the initial transversality condition, note that \eqref{eq:S04g=0} has a constraint on the initial state given by $\bx_0 - \bx^\sharp_a = \bzero$. Let $\widetilde{\bpsi}_0 \in \real{N_x}$ be a multiplier associated with this constraint.  Then, the endpoint Lagrangian (ignoring the constraint $\phi_0(s_a, t_0) = 0$) is given by,
\begin{equation}\label{eq:Ebar4g=0-2}
\widetilde{E}(\widetilde{\bpsi}_0; \bx_0, t_0, s_a, \bx^\sharp_a, s^\sharp_a; \bx_f, t_f,  s_b, \bx^\sharp_b, s^\sharp_b) :=     \nu^0\, E(\bx_0, t_0, s_a, \bx^\sharp_a, s^\sharp_a; \bx_f, t_f,  s_b, \bx^\sharp_b, s^\sharp_b)
 + \widetilde{\bpsi}_0 (\bx_0 - \bx^\sharp_a)
\end{equation}
Applying Theorem~\ref{theorem:TVC-alg} to \eqref{eq:Ebar4g=0-2}, the initial transversality condition is computed as,
\begin{equation}\label{eq:TVC4g=0}
- \blam(t_0) =  \nu^0\, \partial_{\bx_0} E(\bx_0, t_0, \bx_f, t_f; s_a, \bx^\sharp_a, s^\sharp_a, s_b, \bx^\sharp_b, s^\sharp_b)
 + \widetilde{\bpsi}_0
\end{equation}
Comparing \eqref{eq:TVC4g=0} with \eqref{eq:nu0-def}, it follows that these two equations are identical under the interpretation $\widetilde{\bpsi}_0 = \bpsi_0$.

Finally, \eqref{eq:ross-param-a1} reduces to the generalized KKT condition for $ \bx^\sharp_a$ using \eqref{eq:Ebar4g=0-2} as a Lagrangian (with  $(\bx^\sharp_a, s^\sharp_a) \in P_a$ being accounted for via $\etab_{x^\sharp_a}$). See also Lemma~\ref{lemma:TVC} in Sec.~\ref{sec:proof} for additional context.

The preceding discussions can be summarized in terms of the following corollary to Theorem~\ref{theorem:ross}:
\begin{corollary}\label{corr:1}
Suppose the boundary differential equations in the definition of Problem~$(\bP^{ode+})$ vanish. Then the  constraints on the boundary points reduce to the ``static'' specifications on the boundary parameters $\bp_{a}$ and $\bp_b$ defined in \eqref{eq:pa+pb-constr}.  In this case, the transversality conditions given by Theorem~\ref{theorem:ross} simplify to equivalent conditions that can be obtained independently by an application of Theorem~\ref{theorem-TVC-2}.
\end{corollary}

Stated differently, Corollary~\ref{corr:1} simply states that if the boundary conditions do not contain differential equations, then Theorem~\ref{theorem:ross} generates results that can be gotten by a direct application of Theorem~\ref{theorem-TVC-2}.  In this context, Corollary~\ref{corr:1} is essentially an independent validation of Theorem~\ref{theorem:ross}.

\subsection{Case 2: Uncoordinated Boundary Clock Times}
The uncoordinated boundary clock times correspond to the case of $\phi_0(s_a, t_0) \equiv 0 \equiv \phi_f(s_b, t_f)$.  See Example~\ref{ex:1}, Fig.~\ref{fig:unitCircle} and Definition~\ref{def:phantom} for context.  As in Case~$1$ of the previous subsection, we analyze the problem for initial-time conditions only.  Results for the final-time conditions follow similarly.

With $\phi_0 \equiv 0$, \eqref{eq:ross-hvc-0} reduces to the established Hamiltonian value condition\cite{ross-book} while \eqref{eq:ross-tvc-lam0} simplifies to,
\begin{equation}\label{eq:cor-2-start}
\bpsi_0^T \bg_0(\bx_0, s_a) =
 \nu^0\, \partial_{s_a}E(\bx_0, t_0, s_a, \bx^\sharp_a, s^\sharp_a; \bx_f, t_f,  s_b, \bx^\sharp_b, s^\sharp_b)
 + \eta_{s_a}+\sigma_{s_a}
\end{equation}
Suppose we assume the local time-like variable $s_a$  is not included in the cost function; i.e., neither  $E$ nor $F$ are functions of $s_a$. Then the first and last term on the right-hand-side of \eqref{eq:cor-2-start} vanish. In this case, we get the condition,
\begin{equation}\label{eq:cor-2-start-2}
\bpsi_0^T \bg_0(\bx_0, s_a) = \eta_{s_a}
\end{equation}
If we now assume $Q_a = \Real$, then $\eta_{s_a} = 0$.  As a result, \eqref{eq:cor-2-start-2} generates the interesting condition,
\begin{equation}\label{eq:cor-2-wow-1}
\bpsi_0^T \bg_0(\bx_0, s_a) = 0
\end{equation}
That is, under the assumptions stated earlier, we have a transversality condition that requires the weak adjoint covector $\bpsi_0$ to be orthogonal (see Remark~\ref{rem:ortho}) to the vector field $\bg_0(\cdot)$ at a candidate optimal point $(\bx_0, s_a)$.  We capture this result in the form of the following corollary:
\begin{corollary}\label{corr:2}
Suppose the boundary clock times in the definition of Problem~$(\bP^{ode+})$ are uncoordinated and unconstrained. That is, let $\phi_0(s_a, t_0) \equiv 0 \equiv \phi_f(s_b, t_f)$ so that the sets $S_0^{ode+}$ and $S_f^{ode+}$ simplify to,
\begin{multline*}
S_0^{ode+} := \Big\{(\bx_0, t_0, s_a, \bx^\sharp_a, s^\sharp_a):\ \bx_0 = \bx_0(s_a), \quad \bx_0'(s_a) = \bg_0(\bx_0(s_a), s_a), \quad \bx_0(s^\sharp_a) = \bx^\sharp_a, \\
(\bx^\sharp_a, s^\sharp_a) \in P_a, \quad t_0 \in I_0, \quad s_a \in \Real \Big\}
\end{multline*}
\begin{multline*} 
S_f^{ode+} := \Big\{(\bx_f, t_f, s_b, \bx^\sharp_b, s^\sharp_b):\  \bx_f = \bx_f(s_b), \quad \bx_f'(s_b) = \bg_f(\bx_f(s_b), s_b), \quad \bx_f(s^\sharp_b) = \bx^\sharp_b, \\
(\bx^\sharp_b, s^\sharp_b) \in P_b, \quad t_f \in I_f, \quad s_b \in \Real \Big\}
\end{multline*}
Suppose further that the endpoint cost and running cost functions are independent of the boundary clock times $s_a$ and $s_b$.  Then the transversality and Hamiltonian value conditions given by \eqref{eq:ross-tvc-lam0+f} and \eqref{eq:ross-hvc} in Theorem~\ref{theorem:ross} simplify to,
\begin{subequations}\label{eq:corr-2-0+f}
\begin{align}
\bpsi_0^T \bg_0(\bx_0, s_a) &= 0  & \mathcal{H}[@t_0] & = \nu^0\, \partial_{t_0}E(\bx_0, t_0, s_a, \bx^\sharp_a, s^\sharp_a; \bx_f, t_f,  s_b, \bx^\sharp_b, s^\sharp_b) + \nu_{t_0}\label{eq:corr-2-0}\\
\bpsi_f^T \bg_f(\bx_f, s_b) &= 0  & -\mathcal{H}[@t_f] &=  \nu^0\,\partial_{t_f}E(\bx_0, t_0, s_a, \bx^\sharp_a, s^\sharp_a; \bx_f, t_f,  s_b, \bx^\sharp_b, s^\sharp_b) + \nu_{t_f}
\end{align}
\end{subequations}
\end{corollary}
%
\begin{remark}\label{remark:cor-2-isnot}
  The orthogonality conditions $\bpsi_0^T \bg_0(\bx_0, s_a) = 0$ and  $\bpsi_f^T \bg_f(\bx_f, s_b) = 0$ are not the same as those covered in typical textbooks\cite{longuski,brysonHo,ross-book,vinter}.  This is because weak adjoint covectors $\bpsi_0$ and $\bpsi_f$  are not $\blam_0$ and $\blam_f$ respectively; see Eqs.~(\ref{eq:nu0-def}) and (\ref{eq:nuf-def}).  Furthermore, $\bg_0(\bx_0, s_a)$ and $\bg_f(\bx_f, s_b)$ are not algebraic equations that define the boundary conditions; rather, they are vector fields of the differential equations that define the endpoint conditions.
\end{remark}
It is instructive to demonstrate Corollary~\ref{corr:2} and the comments of Remark~\ref{remark:cor-2-isnot} via a simple example.
%
%
\begin{example}\label{ex:2}
Suppose $\bx = (x, y) \in \real{2}$ and we require $\bx_0 = \bx(t_0)$ in an optimal control problem to be selected from the set,
\begin{multline}\label{eq:ex:2-S0}
S_0^{ode+} := \Big\{(\bx_0, t_0) \in \real{2} \times \Real, s_a \in \Real, (\bx^\sharp_a, s^\sharp_a) \in P_a \subset \real{2}\times \Real:\ \bx_0 = (x_0(s_a), y_0(s_a), \quad x_0'(s_a) =y_0(s_a),\\
 y_0'(s_a) =-x_0(s_a), \quad \bx_0(s^\sharp_a) = \bx^\sharp_a \Big\}
\end{multline}
Suppose further that the endpoint cost function $E$ of this problem is independent of $\bx_0$.  Then from \eqref{eq:nu0-def} we have,
\begin{equation}\label{eq:ex-2-nu0}
- \bpsi_0 = \blam(t_0) := [\lambda_x(t_0), \lambda_y(t_0)]^T
\end{equation}
According to \eqref{eq:corr-2-0}, the initial transversality conditions (using \eqref{eq:ex-2-nu0}) are given by,
\begin{equation}\label{eq:ex-2-cor2}
\lambda_x(t_0) y_0(s_a) - \lambda_y(t_0)x_0(s_a) = 0
\end{equation}
and
$$\mathcal{H}[@t_0]  = \nu^0\, \partial_{t_0}E(\bx_0, t_0, s_a, \bx^\sharp_a, s^\sharp_a; \bx_f, t_f,  s_b, \bx^\sharp_b, s^\sharp_b) $$
Obviously \eqref{eq:ex-2-cor2} is a new condition based on Corollary~\ref{corr:2}. The specific vector field in this example is given by,
$$\bg(x, y) = \left[
                \begin{array}{c}
                  y \\
                  -x \\
                \end{array}
              \right]
  $$
Equation~(\ref{eq:ex-2-cor2}) states that $\blam(t_0)$ is orthogonal to $\bg(x,y)$ at the point $(x,y)$. See Remark~\ref{rem:ortho} for context with respect to orthogonality of vectors in different spaces. A geometric description of this result is illustrated in Fig.~\ref{fig:ex-2:VecFieldTVC}.
%
\begin{figure}[h!]
      \centering
      {\parbox{0.9\columnwidth}{
      \centering
      {\includegraphics[width = 0.4\columnwidth]{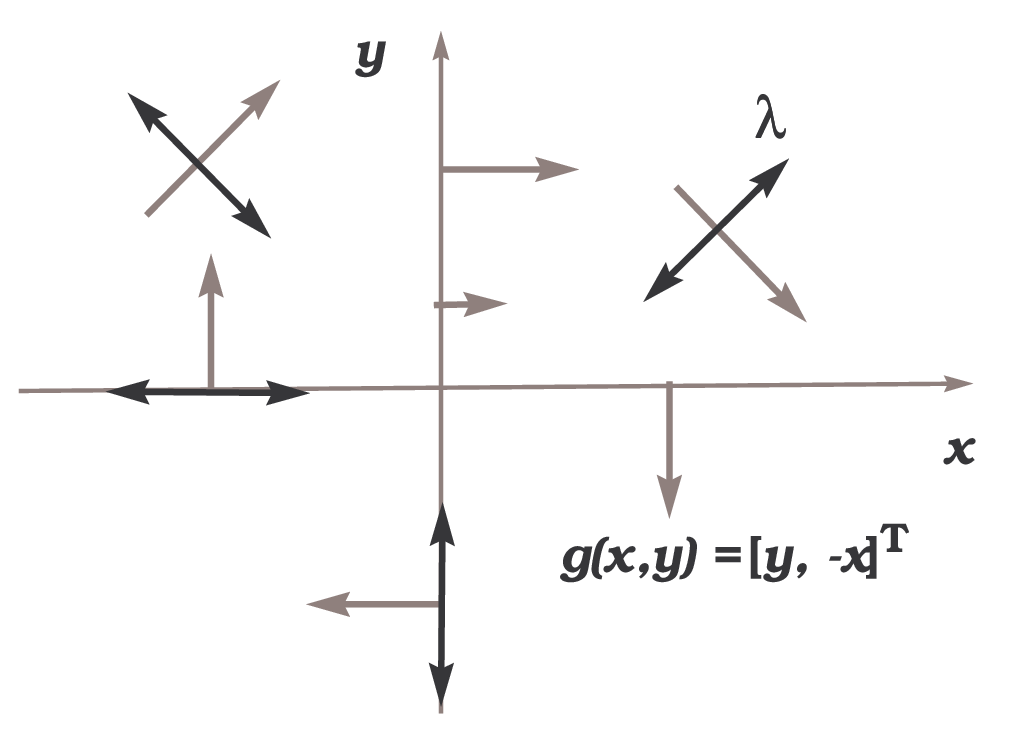}}
      \caption{{An illustration of the orthogonality condition given by \eqref{eq:ex-2-cor2}.}}\label{fig:ex-2:VecFieldTVC}
      }
      }
\end{figure}
%
The single-headed arrows in Fig.~\ref{fig:ex-2:VecFieldTVC} denote the vector field $\bg(x,y)$ and the double-headed arrows denote the initial costate (as a proxy vector\cite{ross-book} to the covector).
\end{example}
Example~\ref{ex:2} illustrates and amplifies the comments in Remark~\ref{remark:cor-2-isnot}.  In the next example, we establish a connection between Corollary~\ref{corr:2} and the standard transversality conditions. 
%
\begin{example}\label{ex:3}
Reconsider Example~\ref{ex:2}.  In this instance, we use the integrability of the differential equations in \eqref{eq:ex:2-S0}:
\begin{equation}\label{eq:ex-2:ode=circ}
x_0'(s_a) =y_0(s_a), \ y_0'(s_a) =-x_0(s_a) \Rightarrow x_0^2(s_a) + y_0^2(s_a) = \text{constant}
\end{equation}
Thus, the differential equations in \eqref{eq:ex:2-S0} are nothing more than differential parameterizations of a ``circle manifold'' such as the one shown in Fig.~\ref{fig:unitCircle}.  Applying the transversality condition for the algebraic transformation of $S_0^{ode+}$; i.e., using \eqref{eq:ex-2:ode=circ} and Theorem~\ref{theorem:TVC-alg}, we get,
\begin{equation}\label{eq:ex-2:tvc-alg}
\lambda_x(t_0) = -\nu_a\, x_0(s_a), \ \lambda_y(t_0) = -\nu_a\, y_0(s_a)
\end{equation}
where $\nu_a$ is an endpoint multiplier associated with the algebraic equation that describes the circle constraint in \eqref{eq:ex-2:ode=circ}.  Eliminating $\nu_a$ from \eqref{eq:ex-2:tvc-alg} yields \eqref{eq:ex-2-cor2}.  Thus the transversality conditions generated by Theorem~\ref{theorem:TVC-alg} and Corollary~\ref{corr:2} are equivalent as indeed they should be.  But note some nuanced differences:
\begin{enumerate}
\item The multipliers in the differential and algebraic parameterizations of the boundary conditions are not the same (i.e., $\nu_a$ is not the same as $\bpsi_0$). In fact, they are not even the same dimensionally. %
\item Equation~(\ref{eq:ex-2:tvc-alg}) is also a statement of orthogonality; however, it is defined in terms of the manifold and not the vector field.  This point is illustrated in Fig.~\ref{fig:ex-2:circleTVC}.
%
\begin{figure}[h!]
      \centering
      {\parbox{0.9\columnwidth}{
      \centering
      {\includegraphics[width = 0.4\columnwidth]{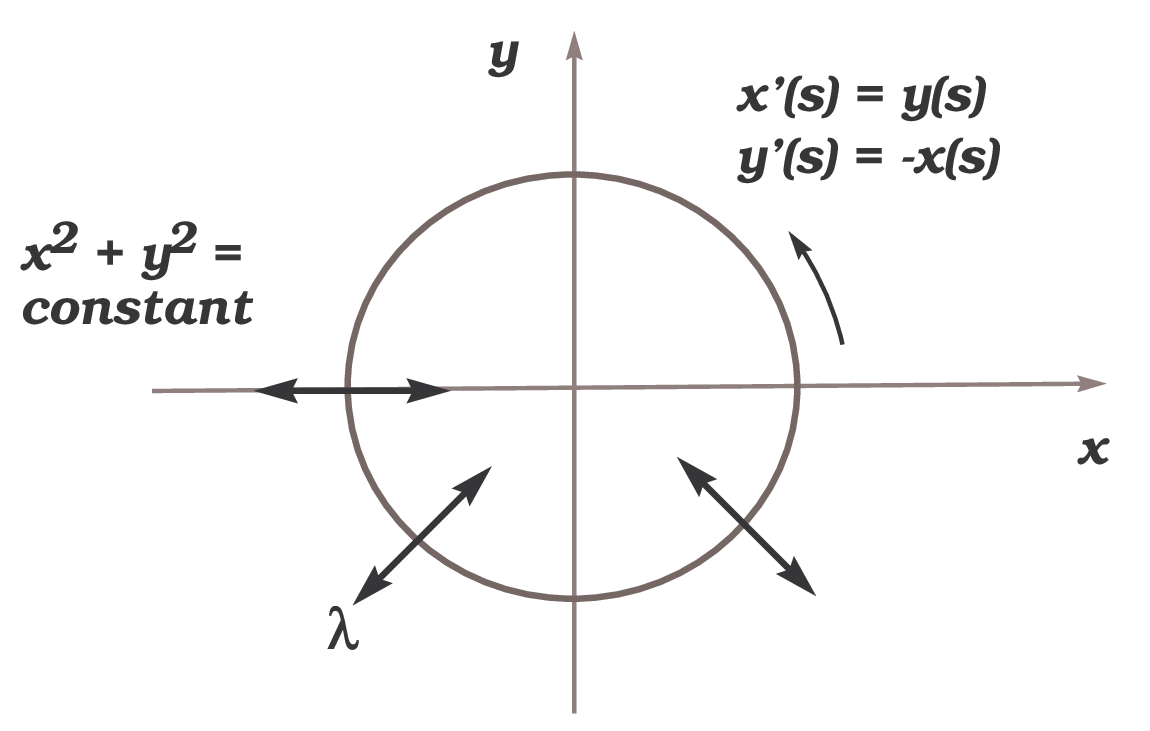}}
      \caption{{Illustrating various connections in the transversality conditions associated with Example~\ref{ex:3}.}}\label{fig:ex-2:circleTVC}
      }
      }
\end{figure}
%
Also shown in Fig.~\ref{fig:ex-2:circleTVC} is the fact that the differential equations that define $S_0$ given by \eqref{eq:ex:2-S0} naturally evolve along a circle.  That is, these differential equations are part of the ``special set'' that are indeed integrable.
\end{enumerate}
\end{example}
The equivalence of \eqref{eq:ex-2-cor2} and \eqref{eq:ex-2:tvc-alg} validates Corollary~\ref{corr:2} and hence Theorem~\ref{theorem:ross}.  In principle, this equivalence should not be a surprise because the uncoordinated clock times simply correspond to a parametrization of the boundary manifold in terms of differential equations with phantom initial conditions (see Definition~\ref{def:phantom}) wherein the independent variable $s_a$ is a phantom or dummy time variable. Note however that Corollary~\ref{corr:2} is more general than a mere equivalence to Theorem~\ref{theorem:TVC-alg}. The former is more general and holds even when the boundary differential equations are not integrable.  Practical examples of this case are boundary conditions in the restricted three-body problem\cite{Farq-1970,szebehely_theory_1967}.
%
\begin{remark}
It is apparent from the discussions thus far that boundary conditions specified in terms of differential equations are indeed more ``natural'' and general than those specified in terms of algebraic equations.
\end{remark}

\subsection{Case 3: Synchronized Boundary Clock Times}
The most common situations of synchronized boundary clock times in aerospace engineering are those corresponding to a rendezvous and intercept\cite{brysonHo,ross-book,longuski}.  In cases that are neither rendezvous nor intercept, the boundary clock times may still be synchronized such as the cases implied in Fig.~\ref{fig:boundaryOrbits}.  In such cases, the boundary clock times are synchronized because the differential equations in the elliptic restricted three-body problem are nonautonomous\cite{Farq-1970,szebehely_theory_1967}.

As in Cases~$1$ and $2$, we discuss the problem for initial-time conditions only.  Results for the final-time conditions follow similarly.  In this spirit, using Definition~\ref{def:sync-bc}, we have
$$\phi_0(s_a, t_0) = s_a - t_0 \quad ( = 0 )$$
As a result we have,
\begin{equation}\label{eq:cor3:parsa=-1}
\partial_{s_a}\phi_0(s_a, t_0) = 1 = - \partial_{t_0}\phi_0(s_a, t_0)
\end{equation}
Furthermore, because $s_a$ is now the same as  $t_0$, where the latter variable is already constrained by \eqref{eq:I0Ifdef}, we set $s_a$ to be an unconstrained (i.e., $Q_a = \Real$) dummy variable that has no impact on the endpoint cost function.  Assume further that the running cost does not depend upon $s_a$.  Taken together, all of these assumptions imply,
\begin{subequations}\label{eq:cor3:sa=nada}
\begin{align}
\partial_{s_a}E(\bx_0, t_0, s_a, \bx^\sharp_a, s^\sharp_a; \bx_f, t_f,  s_b, \bx^\sharp_b, s^\sharp_b) &= 0 \quad \forall\ s_a \in Q_a = R\\
\sigma_{s_a} &= 0 \\
\eta_{s_a} &= 0
\end{align}
\end{subequations}
Substituting \eqref{eq:cor3:parsa=-1} and \eqref{eq:cor3:sa=nada} in \eqref{eq:ross-tvc-lam0}  and \eqref{eq:ross-hvc-0} yields,
\begin{subequations}
\begin{align}
 \bpsi_0^T \bg_0(\bx_0, s_a) &=  \nu_a^C  \label{eq:corr3-proof-1a}\\
\mathcal{H}[@t_0] - \nu^0\, \partial_{t_0}E(\bx_0, t_0, \bx^\sharp_a, s^\sharp_a; \bx_f, t_f,  s_b, \bx^\sharp_b, s^\sharp_b) - \nu_{t_0} &= - \nu_a^C \label{eq:corr3-proof-1b}
\end{align}
\end{subequations}
Adding \eqref{eq:corr3-proof-1a} and \eqref{eq:corr3-proof-1b} eliminates $\nu_a^C$ resulting in the following equation:
$$ \mathcal{H}[@t_0] - \nu^0\, \partial_{t_0}E(\bx_0, t_0, \bx^\sharp_a, s^\sharp_a; \bx_f, t_f,  s_b, \bx^\sharp_b, s^\sharp_b) - \nu_{t_0} + \bpsi_0^T \bg_0(\bx_0, s_a) = 0  $$
Hence, we have the following corollary:
\begin{corollary}\label{corr:3}
Suppose the boundary clock times in the definition of Problem~$(\bP^{ode+})$ are synchronized according to Definition~\ref{def:sync-bc}. Let  $Q_a = \Real = Q_b$ in \eqref{eq:Sode+=bydef}.  Suppose that the endpoint and running cost functions are not dependent on the local time-like variables $s_a$ and $s_b$.
Then the transversality and Hamiltonian value conditions given by \eqref{eq:ross-tvc-lam0+f} and \eqref{eq:ross-hvc} in Theorem~\ref{theorem:ross} reduce to,
\begin{multline}\label{eq:corr3:tvc-0}
\mathcal{H}[@t_0] - \nu^0\,\partial_{t_0}E(\bx_0, t_0, x^\sharp_a, s^\sharp_a; \bx_f, t_f, \bx^\sharp_b, s^\sharp_b)  - \nu_{t_0} \\
= [\bg_0(\bx_0, t_0)]^T \left[\blam(t_0) + \nu^0\,\partial_{\bx_0} E(\bx_0, t_0, x^\sharp_a, s^\sharp_a; \bx_f, t_f, \bx^\sharp_b, s^\sharp_b)   \right]
\end{multline}
\begin{multline}\label{eq:corr3:tvc-f}
\mathcal{H}[@t_f] + \nu^0\,\partial_{t_f} E(\bx_0, t_0, x^\sharp_a, s^\sharp_a; \bx_f, t_f, \bx^\sharp_b, s^\sharp_b)   + \nu_{t_f} \\
= [\bg_f(\bx_f, t_f)]^T \left[\blam(t_f) - \nu^0\,\partial_{\bx_f} E(\bx_0, t_0, x^\sharp_a, s^\sharp_a; \bx_f, t_f, \bx^\sharp_b, s^\sharp_b)   \right]
\end{multline}
\end{corollary}
Although \eqref{eq:corr3:tvc-0} and \eqref{eq:corr3:tvc-f} do not contain the symbols $s_a$ and $s_b$, it may be preferable to retain them in the definitions of $S_0^{ode+}$ and $S_f^{ode+}$ because they are dummy variables equivalent to $t_0$ and $t_f$ respectively.  Keeping $s_a$ and $s_b$ separate from $t_0$ and $t_f$ allows us to retain the symbols $\bx_0'(s_a)$ and $\bx_f'(s_b)$ respectively. Otherwise, $S_0^{ode+}$ and $S_f^{ode+}$ would need to be redefined with the symbols $\bx_0'(t_0)$ and $\bx_f'(t_f)$ to represent ``time-derivatives'' with respect to the initial and final-time variables $t_0$ and $t_f$, leading to the awkward and confusing notation,
\begin{equation}\label{eq:awkward}
\bx_0' \equiv \frac{d\bx_0}{d t_0} \quad\text{and}\quad \bx_f' \equiv \frac{d\bx_f}{d t_f}
\end{equation}
The implication of using $t_0$ and $t_f$ as an independent time variable in \eqref{eq:awkward} is easily prevented by retaining the symbols $s_a$ and $s_b$. 

A simpler version of Corollary~\ref{corr:3} is obtained when $P_a$ and $P_b$ are singletons. This situation is depicted in Fig.~\ref{fig:ODECurves30f}.  
%
\begin{figure}[h!]
      \centering
      {\parbox{0.9\columnwidth}{
      \centering
      {\includegraphics[width = 0.5\columnwidth]{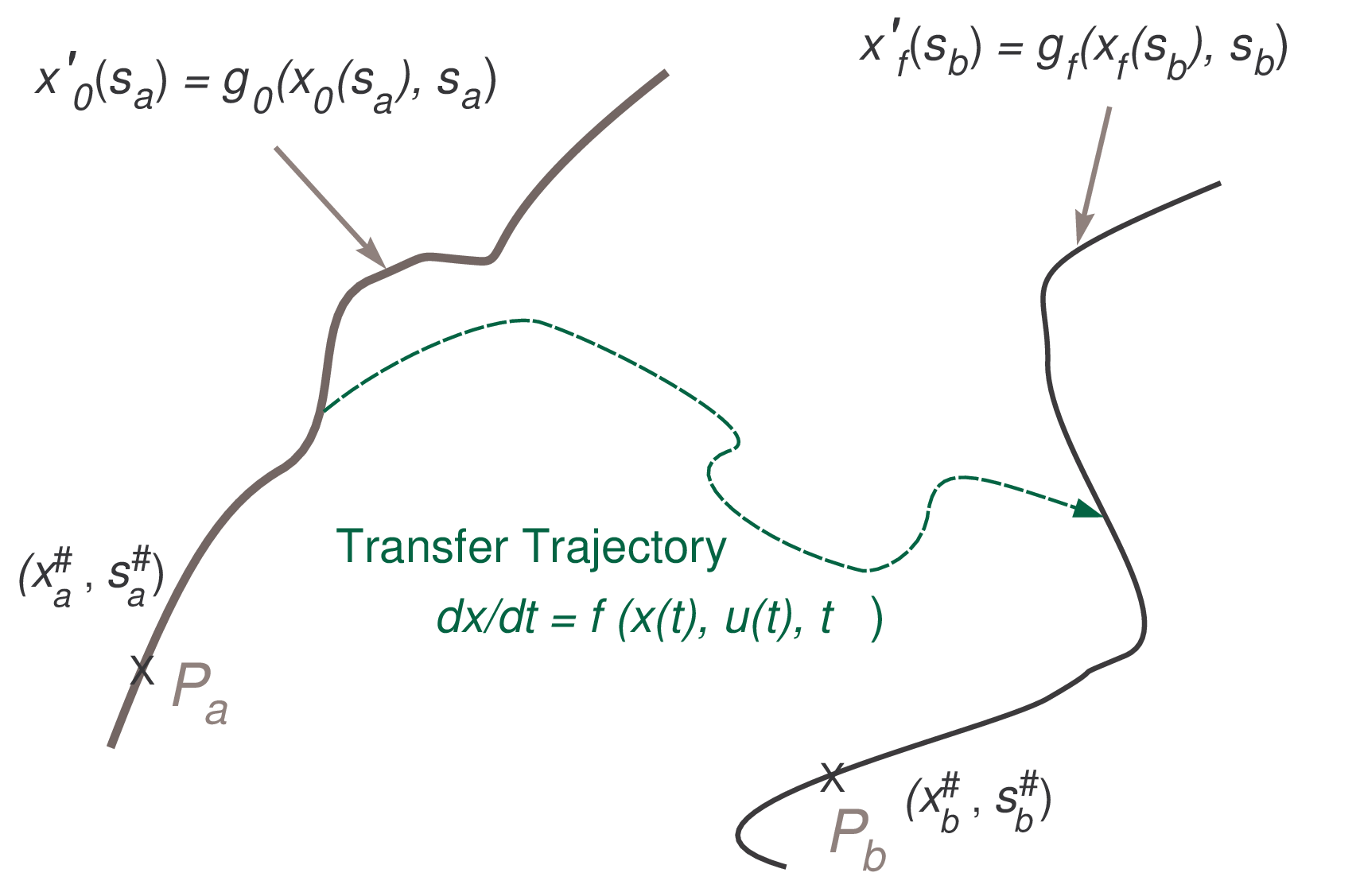}}
      \caption{{Schematic for Corolllary~\ref{corr:4}.}}\label{fig:ODECurves30f}
      }
      }
\end{figure}
%
In the context of the motivating problems discussed in Section~\ref{sec:Intro}, Fig.~\ref{fig:ODECurves30f} may be construed as an abstract version of transferring a spacecraft from one of the orbits shown in Fig.~\ref{fig:boundaryOrbits} to rendezvous with another spacecraft in a different orbit. Note  also that the symbol $t^\sharp_0$ in \eqref{eq:intro-init-0} is the same as $s^\sharp_a$ (with $s^\sharp_a$ fixed). In Section~\ref{sec:Intro} it was more convenient to use the symbol $t^\sharp_0$ because the concept of ``$s$-time'' was not developed until Section~\ref{sec:newTVCs}.  In any case, when $P_a$ and $P_b$ are singletons, then $\bx^\sharp_a, s^\sharp_a, \bx^\sharp_b$ and $s^\sharp_b$ are no longer optimization variables. As a result, \eqref{eq:corr3:tvc-0} and \eqref{eq:corr3:tvc-f} simplify to,
\begin{subequations}\label{eq:corr4-main}
\begin{align}
\mathcal{H}[@t_0] - \nu^0\,\partial_{t_0}E\left(\bx_0, t_0, \bx_f, t_f \right) - \nu_{t_0}
& = [\bg_0(\bx_0, t_0)]^T \left[\blam(t_0) + \nu^0\,\partial_{\bx_0} E\left(\bx_0, t_0, \bx_f, t_f \right) \right]\\
\mathcal{H}[@t_f] + \nu^0\,\partial_{t_f}E\left(\bx_0, t_0, \bx_f, t_f \right) + \nu_{t_f}
& = [\bg_f(\bx_f, t_f)]^T \left[\blam(t_f) - \nu^0\,\partial_{\bx_f} E\left(\bx_0, t_0, \bx_f, t_f \right) \right]
\end{align}
\end{subequations}
Equation~(\ref{eq:corr4-main}) generates a number of interesting results that are summarized in the following corollary:
\begin{corollary}\label{corr:4}
Suppose the boundary conditions in the definition of Problem~$(\bP^{ode+})$ are given by,
\begin{equation*}
S_0^{ode+} := \Big\{(\bx_0, t_0, s_a)  \in \real{N_x} \times \Real \times \Real:\ \bx_0 = \bx_0(s_a), \quad t_0 -s_a = 0,  \quad \bx_0'(s_a) = \bg_0(\bx_0(s_a), s_a), \quad \bx_0(s^\sharp_a) = \bx^\sharp_a \Big\}
\end{equation*}
\begin{equation*}
S_f^{ode+} := \Big\{(\bx_f, t_f, s_b) \in \real{N_x} \times \Real \times \Real:\  \bx_f = \bx_f(s_b), \quad t_f- s_b = 0, \quad \bx_f'(s_b) = \bg_f(\bx_f(s_b), s_b), \quad \bx_f(s^\sharp_b) = \bx^\sharp_b
 \Big\}
\end{equation*}
where $(\bx^\sharp_a, s^\sharp_a)$ and $(\bx^\sharp_b, s^\sharp_b)$ are specified points in $\real{N_x} \times \Real$.  Furthermore, let $t_0$ and $t_f$ be unconstrained; i.e., $t_0 \in \Real$ and $t_f \in \Real$ (with $t_f > t_0$).  Then the following holds:
\begin{enumerate}
\item If the clock times are free with $\partial_{t_0}E\left(\bx_0, t_0, \bx_f, t_f\right) =0 = \partial_{t_f}E\left(\bx_0, t_0, \bx_f, t_f\right)$, then the Hamiltonian value conditions are given by,
\begin{subequations}\label{eq:corr4-time-free-result}
\begin{align}
\mathcal{H}[@t_0] & = \left[\blam(t_0) + \nu^0\,\partial_{\bx_0} E\left(\bx_0, \bx_f \right) \right]^T  \bg_0(\bx_0, t_0)\\
\mathcal{H}[@t_f] & =  \left[\blam(t_f) - \nu^0\,\partial_{\bx_f} E\left(\bx_0, \bx_f \right) \right]^T \bg_f(\bx_f, t_f)
\end{align}
\end{subequations}
\item For minimum transfer-time ``normal'' problems, namely, $E\left(\bx_0, t_0, \bx_f, t_f\right) =t_f - t_0$ and $\nu^0 = 1$, the Hamiltonian value conditions are given by,
\begin{subequations}
\begin{align}
\mathcal{H}[@t_0] +1
& = \blam^T(t_0)\, \bg_0(\bx_0, t_0)\\
\mathcal{H}[@t_f] + 1
& = \blam^T(t_f) \, \bg_f(\bx_f, t_f)
\end{align}
\end{subequations}
\item If the problem has no endpoint cost (i.e., $E \equiv 0$ and the cost functional is given purely in terms of the running cost) then the Hamiltonian value conditions are given quite simply by,
\begin{subequations}
\begin{align}
\mathcal{H}[@t_0] & = \blam^T(t_0)\, \bg_0(\bx_0, t_0)\\
\mathcal{H}[@t_f] & = \blam^T(t_f) \, \bg_f(\bx_f, t_f)
\end{align}
\end{subequations}
\end{enumerate}
\end{corollary}
According to Corollary~\ref{corr:4}, time-free problems do not necessarily exhibit the familiar condition $\mathcal{H}[@t_f] = 0$.  Note however that this statement in Corollary~\ref{corr:4} is not a contradiction because \eqref{eq:corr4-time-free-result} applies to synchronized clock times whereas the the better-known condition of  $\mathcal{H}[@t_f] = 0$ applies to uncoordinated clock times.  Indeed, from Corollary~\ref{corr:2}, \eqref{eq:corr-2-0+f} we get $\mathcal{H}[@t_f] = 0$ if $\partial_{t_f}E(\bx_0, t_0, \bx_f, t_f;  \bx^\sharp_a, s^\sharp_a, \bx^\sharp_b, s^\sharp_b) = 0$ and  $ \nu_{t_f} = 0$.  For the same reason, minimum transfer time problems in Corollary~\ref{corr:4} do not exhibit the more familiar condition that the final value of the lower Hamiltonian is $-1$.

\section{A Proof of Theorem~\ref{theorem:ross} with Insights on  $\bpsi_0$ and $\bpsi_f$}\label{sec:proof}
The proof of Theorem~\ref{theorem:ross} involves several steps.  We use three lemmas to organize these steps. These lemmas also provide further insights on how the weak adjoint covectors $\bpsi_0$ and $\bpsi_f$ appear quite naturally.

\subsection{Development of the Three Lemmas Required for a Proof of Theorem~\ref{theorem:ross}}

In the first lemma, we rewrite $S_0$ and $S_f$ in a form that is more conducive to the application of Theorems~\ref{theorem-TVC-2} and \ref{theorem:TVC-alg}.
%
\begin{lemma}\label{lemma:equiv}
The sets $S_0^{ode+}$ and $S_f^{ode+}$ given by \eqref{eq:S0=def} and \eqref{eq:Sf=def} respectively, may be written equivalently as,
\begin{subequations}\label{eq:S0+Sf-equiv}
\begin{multline}\label{eq:S0=equiv}
S_0^\dag := \Big\{(\bx_0, t_0, s_a, \bx^\sharp_a, s^\sharp_a):\ \bx_0- \bx^\sharp_a - \int^{s_a}_{s^\sharp_a} \bg_0(\bx_0(s), s)\, ds = \bzero,  \\
(\bx^\sharp_a, s^\sharp_a) \in P_a, \quad
\phi_0(s_a,t_0)=0,  \quad t_0 \in I_0, \quad s_a \in Q_a \Big\}
\end{multline}
\begin{multline}\label{eq:Sf=equiv}
S_f^\dag := \Big\{(\bx_f, t_f, s_b, \bx^\sharp_b, s^\sharp_b):\ \bx_f - \bx^\sharp_b - \int^{s_b}_{s^\sharp_b} \bg_f(\bx_f(s), s)\, ds = \bzero, \\
(\bx^\sharp_b, s^\sharp_b) \in P_b, \quad
\phi_f(s_b, t_f)=0,  \quad t_f \in I_f, \quad s_b \in Q_b \Big\}
\end{multline}
\end{subequations}
\end{lemma}
\begin{proof}
To prove the equivalence between \eqref{eq:S0=def} and \eqref{eq:S0=equiv}, we first rewrite the differential equation $\bx_0'(s_a) = \bg_0(\bx_0(s_a), s_a)$ in its equivalent ``integral'' form\cite{teschl-ode,hale-ode},
\begin{equation}\label{eq:x0-ode-integ-form}
\bx_0(s_a) = \bx^\sharp_a + \int^{s_a}_{s^\sharp_a} \bg_0(\bx_0(s), s)\, ds
\end{equation}
By definition (see \eqref{eq:S0=def}) $\bx_0 = \bx_0(s_a)$.  Substituting $\bx_0$ for $\bx_0(s_a)$ in \eqref{eq:x0-ode-integ-form} generates the constraint equation,
\begin{equation}\label{eq:x0-ode-integ-constraint}
\bx_0- \bx^\sharp_a - \int^{s_a}_{s^\sharp_a} \bg_0(\bx_0(s), s)\, ds = \bzero
\end{equation}
This equation is equivalent to three equations in the definition of $S_0^{ode+}$ given by \eqref{eq:S0=def}, namely, $\bx_0 = \bx_0(s_a)$, $ \bx_0'(s_a) = \bg_0(\bx_0(s_a), s_a)$ and  $\bx_0(s^\sharp_a) = \bx^\sharp_a$.  Replacing these three equations with \eqref{eq:x0-ode-integ-constraint} generates \eqref{eq:S0=equiv}. A proof of the equivalence between \eqref{eq:Sf=def} and \eqref{eq:Sf=equiv} follows similarly.
\end{proof}

From Lemma~\ref{lemma:equiv}, it follows that the set $S^{ode+}= S_0^{ode+} \times S_f^{ode+}$ defined in \eqref{eq:Sode+=bydef} is some combination of the endpoint sets defined in Theorems~\ref{theorem-TVC-2} and \ref{theorem:TVC-alg}. Hence, we now combine the most relevant components of Theorems~\ref{theorem-TVC-2} and \ref{theorem:TVC-alg} to construct a lemma that directly supports a proof of Theorem~\ref{theorem:ross}.
%
\begin{lemma}\label{lemma:TVC}
Suppose Problem~($\widetilde{\bP}$) denotes the special case of Problem~$(\bP)$ where $S$ is given by a combination of algebraic inequalities and an abstract set according to,
\begin{equation}\label{eq:S=hyb}
S = \widetilde{S} := \set{(\bx_0, t_0, \bx_f, t_f, \widetilde{\bp}):\ \be^L \le \be(\bx_0, t_0, \bx_f, t_f, \widetilde{\bp}) \le \be^U,\ \widetilde{\bp} \in S_{\widetilde{p}} \subset \real{N_{\widetilde{p}}} }
\end{equation}
Let $\overline{E}(\bnu, \bx_0, t_0, \bx_f, t_f, \widetilde{\bp})$  be defined by \eqref{eq:Ebar=bydef} with $\bp$ replaced by $\widetilde{\bp}$ and $\bnu$ satisfying \eqref{eq:complementarity}.
Then the transversality conditions for Problem~$(\widetilde{\bP})$ is given by,
\begin{equation}\label{eq:TVC-hyb}
\big(-\blam(t_0), \mathcal{H}[@t_0], \blam(t_f), -\mathcal{H}[@t_f], -\bsigma \big) = \partial_{(\bx_0, t_0, \bx_f, t_f, \bp)} \overline{E}(\bnu, \bx_0, t_0, \bx_f, t_f, \widetilde{\bp}) + (\bzero, \etab_{\widetilde{p}})
\end{equation}
where  $(\bzero, \etab_{\widetilde{p}})$ is an $(2(N_x+1) + N_{\widetilde{p}})$-dimensional vector whose first $(2N_x+2)$ components are zero and  $\etab_{\widetilde{p}} \in N_{S_{\widetilde{p}}}(\widetilde{\bp})$.
\end{lemma}
\begin{proof}
This lemma is a direct consequence of Theorems~\ref{theorem-TVC-2} and \ref{theorem:TVC-alg}.  Complete details of this ``multiplier rule'' derived from first-principles are given in \cite{vinter,clarke-2013book}.
\end{proof}
\begin{remark}\label{rem:Ebar=algOnly}
The ``endpoint Lagrangian'' $\overline{E}$ in \eqref{eq:TVC-hyb} corresponds to only that portion of the Lagrangian that incorporates just the algebraic constraints. The ``full'' endpoint Lagrangian is given by\cite{vinter,clarke-classic-1990},
\begin{equation}
 \overline{E}(\bnu, \bx_0, t_0, \bx_f, t_f, \widetilde{\bp}) + K d_{S_{\widetilde{p}}}(\widetilde{\bp})
\end{equation}
where $d_{S_{\widetilde{p}}}(\widetilde{\bp})$ is the distance function to the set $S_{\widetilde{p}}$ at the point $\widetilde{\bp}$ and $K > 0$ is a sufficiently large number. The distance function is nonsmooth\cite{vinter,clarke-classic-1990}. Its generalized derivative is a set which turns out to be the normal cone (under certain technical conditions whose discussions are beyond the scope of this paper).  See Theorem 4.8.5 in \cite{vinter} and Proposition 2.4.2 and Theorem 6.1.1 in \cite{clarke-classic-1990}.
\end{remark}
Finally, we need the following basic lemma to shorten the proof of Theorem~\ref{theorem:ross}.
%
\begin{lemma}\label{lemma:neat}
The following equations hold:
\begin{subequations}
\begin{align}
\partial_{s_a} \left(\bx_0- \bx^\sharp_a - \int^{s_a}_{s^\sharp_a} \bg_0(\bx_0(s), s)\, ds \right)  &= -\bg_0(\bx_0, s_a) \label{eq:neat-1}\\
\partial_{s^\sharp_a} \left(\bx_0- \bx^\sharp_a - \int^{s_a}_{s^\sharp_a} \bg_0(\bx_0(s), s)\, ds \right) &= \bg_0(\bx^\sharp_a, s^\sharp_a) \\
\partial_{s_b} \left(\bx_f- \bx^\sharp_b - \int^{s_b}_{s^\sharp_b} \bg_f(\bx_f(s), s)\, ds \right)  &= -\bg_f(\bx_f, s_b) \\
\partial_{s^\sharp_b} \left(\bx_f- \bx^\sharp_b - \int^{s_b}_{s^\sharp_b} \bg_f(\bx_f(s), s)\, ds \right) &= \bg_f(\bx^\sharp_b, s^\sharp_b)
\end{align}
\end{subequations}
\end{lemma}
\begin{proof}
To prove \eqref{eq:neat-1}, it follows from basic calculus that the following is true:
\begin{equation}\label{eq:neat-1-proof}
\partial_{s_a} \left(\bx_0- \bx^\sharp_a - \int^{s_a}_{s^\sharp_a} \bg_0(\bx_0(s), s)\, ds \right)  = -\bg_0(\bx_0(s_a), s_a)
\end{equation}
From the original definition of $S_0$ given by \eqref{eq:S0=def}, we have $\bx_0(s_a) = \bx_0$.  Substituting $\bx_0$ for $\bx_0(s_a)$ in \eqref{eq:neat-1-proof} yields \eqref{eq:neat-1}.  The proof of the remainder of the equations follow similarly.
\end{proof}

\subsection{Proof of Theorem~\ref{theorem:ross}}

From Lemma~\ref{lemma:equiv}, we have the equivalence relationship, $S_0^{ode+}\times S_f^{ode+} \equiv S_0^\dagger \times S_f^\dagger$. Treating $P_a, Q_a, P_b$ and $Q_b$ in \eqref{eq:S0=equiv} and \eqref{eq:Sf=equiv} as abstract sets, the algebraic component of the endpoint Lagrangian per Lemma~\ref{lemma:TVC} (see also Remark~\ref{rem:Ebar=algOnly}) corresponding to $S_0^{ode+} \times S_f^{ode+}$ (and denoted as $\widetilde{E}$) can be written as,
\begin{multline}\label{eq:proof-TVC-0}
\widetilde{E}(\nu^0, \widetilde{\bpsi}_0, \nu_a^C, \nu_{t_0},  \widetilde{\bpsi}_f, \nu_b^C, \nu_{t_f},   \bx_0, t_0, \bp_a; \bx_f, t_f, \bp_b) :=
\nu^0\,E(\bx_0, t_0, \bp_a; \bx_f, t_f, \bp_b)\\
+  \widetilde{\bpsi}_0^T \left[\bx_0- \bx^\sharp_a - \int^{s_a}_{s^\sharp_a} \bg_0(\bx_0(s), s)\, ds \right]
+ \nu_a^C\, \phi_0(s_a,t_0) + \nu_{t_0}\, t_0\\
+\widetilde{\bpsi}_f^T \left[\bx_f - \bx^\sharp_b - \int^{s_b}_{s^\sharp_b} \bg_f(\bx_f(s), s)\, ds \right]
+ \nu_b^C\, \phi_f(s_b,t_f) + \nu_{t_f}\, t_f
\end{multline}
where, $\widetilde{\bpsi}_0 \in \real{N_x}$  and  $\widetilde{\bpsi}_f \in \real{N_x}$ are the multipliers associated with imposing the integral constraints defined in \eqref{eq:S0=equiv} and \eqref{eq:Sf=equiv} respectively.
Applying Lemma~\ref{lemma:TVC} for the first $(2N_x + 2)$ components of \eqref{eq:TVC-hyb} (i.e., the costates and Hamiltonian values at initial and final times), we get,
\begin{subequations}\label{eq:proof-TVC-1}
\begin{align}
-\blam(t_0) &= \nu^0\,\partial_{\bx_0} E(\bx_0, t_0, \bp_a; \bx_f, t_f, \bp_b) + \widetilde{\bpsi}_0 \label{eq:nu0-sol}\\
\mathcal{H}[@t_0] &= \nu^0\, \partial_{t_0}E(\bx_0, t_0, \bp_a; \bx_f, t_f, \bp_b) + \nu_a^C \partial_{t_0}\phi_0(s_a, t_0) + \nu_{t_0} \label{eq:proof-hvc-0}\\
\blam(t_f) &= \nu^0\, \partial_{\bx_f}E(\bx_0, t_0, \bp_a; \bx_f, t_f, \bp_b) + \widetilde{\bpsi}_f \label{eq:nuf-sol}\\
-\mathcal{H}[@t_f] &= \nu^0\,\partial_{t_f}E(\bx_0, t_0, \bp_a; \bx_f, t_f, \bp_b)  + \nu_b^C \partial_{t_f}\phi_f(s_b, t_f) + \nu_{t_f} \label{eq:proof-hvc-f}
\end{align}
\end{subequations}
Comparing \eqref{eq:nu0-sol} and \eqref{eq:nuf-sol} to \eqref{eq:nu0-def} and \eqref{eq:nuf-def} respectively, it follows that,
\begin{equation}\label{eq:nutilde=nu}
 \widetilde{\bpsi}_0 = \bpsi_0  \quad\text{and} \quad \widetilde{\bpsi}_f = \bpsi_f
\end{equation}
Furthermore, Eqs~(\ref{eq:proof-hvc-0}) and (\ref{eq:proof-hvc-f}) are indeed \eqref{eq:ross-hvc-0} and \eqref{eq:ross-hvc-f} respectively.
Applying Lemma~\ref{lemma:TVC} for the optimization variables $s_a$ and $s_b$, and using Lemma~\ref{lemma:neat}, we get (see \eqref{eq:sigma=bydef-2}),
\begin{subequations}\label{eq:proof-TVC-2}
\begin{align}
-\sigma_{s_a} &= \nu^0\, \partial_{s_a}E(\bx_0, t_0, s_a, \bx^\sharp_a, s^\sharp_a; \bx_f, t_f, s_b, \bx^\sharp_b, s^\sharp_b)
-\widetilde{\bpsi}_0^T\bg_0(\bx_0, s_a)
 + \nu_a^C \partial_{s_a}\phi_0(s_a, t_0) + \eta_{s_a} \label{eq:nu0-prob}\\
-\sigma_{s_b} &= \nu^0\, \partial_{s_b}E(\bx_0, t_0, s_a, \bx^\sharp_a, s^\sharp_a; \bx_f, t_f, s_b, \bx^\sharp_b, s^\sharp_b)
-\widetilde{\bpsi}_f^T\bg_f(\bx_f, s_b)
 + \nu_b^C \partial_{s_b}\phi_f(s_b, t_f) + \eta_{s_b} \label{eq:nuf-prob}
\end{align}
\end{subequations}
Substituting \eqref{eq:nutilde=nu} for $\widetilde{\bpsi}_0$ and $\widetilde{\bpsi}_f$ in  \eqref{eq:nu0-prob} and \eqref{eq:nuf-prob} yields \eqref{eq:ross-tvc-lam0} and \eqref{eq:ross-tvc-lamf} respectively.

Applying Lemma~\ref{lemma:TVC} for the remainder of the optimization variables, namely, $\bx^\sharp_a$, $s^\sharp_a$, $\bx^\sharp_b$ and  $s^\sharp_b$,  and using Lemma~\ref{lemma:neat}, we get,
\begin{subequations}\label{eq:proof-param}
\begin{align}
-\bsigma_{x^\sharp_a} &= \nu^0\, \partial_{\bx^\sharp_a}E(\bx_0, t_0, s_a, \bx^\sharp_a, s^\sharp_a; \bx_f, t_f, s_b, \bx^\sharp_b, s^\sharp_b)
-\widetilde{\bpsi}_0 + \etab_{x^\sharp_a} \\
-\sigma_{s^\sharp_a} &= \nu^0\, \partial_{s^\sharp_a}E(\bx_0, t_0, s_a, \bx^\sharp_a, s^\sharp_a; \bx_f, t_f, s_b, \bx^\sharp_b, s^\sharp_b)
+\widetilde{\bpsi}_0^T\bg_0(\bx^\sharp_a, s^\sharp_a)  + \eta_{s^\sharp_a} \\
-\bsigma_{x^\sharp_b} &= \nu^0\, \partial_{\bx^\sharp_b}E(\bx_0, t_0, s_a, \bx^\sharp_a, s^\sharp_a; \bx_f, t_f, s_b, \bx^\sharp_b, s^\sharp_b)
-\widetilde{\bpsi}_f + \etab_{x^\sharp_b} \\
-\sigma_{s^\sharp_b} &= \nu^0\, \partial_{s^\sharp_b}E(\bx_0, t_0, s_a, \bx^\sharp_a, s^\sharp_a; \bx_f, t_f, s_b, \bx^\sharp_b, s^\sharp_b)
+\widetilde{\bpsi}_f^T\bg_f(\bx^\sharp_b, s^\sharp_b)  + \eta_{s^\sharp_b}
\end{align}
\end{subequations}
Substituting \eqref{eq:nutilde=nu} in \eqref{eq:proof-param} yields \eqref{eq:ross-param}. 

\subsection{Insights on $\bpsi_0$ and $\bpsi_f$ as Weak Adjoint Covectors for the Boundary Differential Equations}

According to \eqref{eq:nutilde=nu} $\bpsi_0$  and $\bpsi_f$ are exactly the same as $\widetilde{\bpsi}_0$ and $\widetilde{\bpsi}_f$.  Hence, $\bpsi_0$  and $\bpsi_f$ are indeed the multipliers associated with imposing the integral constraints defined in \eqref{eq:S0=equiv} and \eqref{eq:Sf=equiv} respectively. Obviously, $\widetilde{\bpsi}_0$ and $\widetilde{\bpsi}_f$, and hence,  $\bpsi_0$  and $\bpsi_f$ are not the usual (i.e., strong) adjoint covectors associated with imposing differential constraints. To better understand these nuances, consider the imposition of the initial boundary differential constraint given by the second term on the right-hand-side of \eqref{eq:proof-TVC-0}.  This term can also be written as,
\begin{equation}\label{eq:interp-1}
  \widetilde{\bpsi}_0^T \left[\bx_0- \bx^\sharp_a - \int^{s_a}_{s^\sharp_a} \bg_0(\bx_0(s), s)\, ds \right] =  \int^{s_a}_{s^\sharp_a}\left(\widetilde{\bpsi}_0\odot \mathbf{1}_{[s^\sharp_a, s_a]}(\xi)\right)^T \left[\bx'_0(\xi) - \bg_0(\bx_0(\xi), \xi) \right]\, d\xi
\end{equation}
where $\mathbf{1}_{[s^\sharp_a, s_a]}(\xi)$ is the indicator function defined in \eqref{eq:indicatorDef}.  The quantity, $\widetilde{\bpsi}_0\odot \mathbf{1}_{[s^\sharp_a, s_a]}(\xi)$, in the right-hand-side of \eqref{eq:interp-1} is exactly the weak adjoint covector defined in Definition~\ref{def:weakAdj}.  Thus, Lemma~\ref{lemma:equiv} is the key result that explains the need for a weak adjoint covector. In principle, we could have proved Theorem~\ref{theorem:ross} without Lemma~\ref{lemma:equiv} by invoking Definition~\ref{def:weakAdj}.  Although this is a valid mathematical route for a proof of Theorem~\ref{theorem:ross}, we deliberately used the alternative path of starting with Lemma~\ref{lemma:equiv} because it naturally motivates \eqref{eq:interp-1} and hence Definition~\ref{def:weakAdj}.

\section{Theorem~\ref{theorem:ross-ext}: A Practical Extension of the Main Result}
Surprisingly, Theorem~\ref{theorem:ross} is not sufficiently general to solve practical problems such as the ones posed in \cite{dixon-diss-2025}.  In \cite{dixon-diss-2025}, additional inequality constraints on $(\bx_0, t_0, \bx_f, t_f)$ are imposed. Suppose we denote these additional constraints abstractly by a set $S^{add}$.  Then to solve practical problems such as the ones posed in \cite{dixon-diss-2025}, \eqref{eq:bc=ode+} must be modified to,
\begin{equation}\label{eq:bc=ode+add}
\big(\bx_0, t_0, \bp_a; \bx_f, t_f, \bp_b; \bp \big) \in  (S_0^{ode+} \times S_f^{ode+})\cap S^{add}
\end{equation}
For \eqref{eq:bc=ode+add} to be meaningful, we must have $(S_0 \times S_f)\cap S^{add} \neq \emptyset$.  That is, $S^{add}$ must not over-constrain the problem.
From Section~\ref{sec:proof} it follows that Theorem~\ref{theorem:ross} can be extended to incorporate \eqref{eq:bc=ode+add} via the normal cone to the set $(S_0^{ode+} \times S_f^{ode+})\cap S^{add}$. The computation of this normal cone is not straightforward for an abstract set $S^{add}$\cite{vinter,clarke-2013book}.  However, if $S^{add} = S^{alg}$ where $S^{alg}$ is the set described by algebraic constraints defined previously by \eqref{eq:S=efun} then from the proof of Theorem~\ref{theorem:ross} it follows that Theorem~\ref{theorem:ross} holds with $E$ replaced by $\overline{E}$ defined in \eqref{eq:Ebar=bydef}.  This result is memorialized in terms of the following extension to Theorem~\ref{theorem:ross}:
%
\begin{theorem}[Extension to Theorem~\ref{theorem:ross}]\label{theorem:ross-ext}
Let the assumptions of Theorem~\ref{theorem:ross} hold. Let Problem~$(\bP^{ode+alg})$  be defined by Problem~$(\bP)$ whose boundary conditions are given by,
\begin{equation}\label{eq:endpt4B++}
\big(\bx_0, t_0, \bp_a; \bx_f, t_f, \bp_b; \bp \big) \in  (S_0^{ode+} \times S_f^{ode+})\cap S^{alg} \neq \emptyset
\end{equation}
where $S^{alg}$ is defined by \eqref{eq:S=efun}.  Let,
\begin{equation}\label{eq:ross-ext-Ebar}
\overline{E}(\bnu; \bx_0, t_0, \bp_a; \bx_f, t_f, \bp_{b}; \bp) :=  \nu^0\,E(\bx_0, t_0, \bp_a; \bx_f, t_f, \bp_{b}; \bp) + \bnu^T \be(\bx_0, t_0, \bx_f, t_f, \bp)
\end{equation}
where $\bnu$ satisfies the complementarity condition given by \eqref{eq:complementarity} and  $(\bp_a; \bp_b):= (s_a, \bx^\sharp_a, s^\sharp_a; s_b, \bx^\sharp_b, s^\sharp_b)$ as before (Cf.~\eqref{eq:pa+pb-constr}).
Let $\bsigma$ be defined by \eqref{eq:sigma=bydef}.
Then the transversality conditions given by Eqs.~(\ref{eq:ross-tvc-lam0+f}), (\ref{eq:ross-hvc}) and (\ref{eq:ross-param}) hold with $E(\cdot)$ replaced by $\overline{E}(\cdot)$ and given explicitly by,
\begin{subequations}\label{eq:ross-ext-tvc-all}
\begin{align}
- \bpsi_0 & =  \blam(t_0) + \partial_{\bx_0} \overline{E}(\bnu; \bx_0, t_0, \bp_a; \bx_f, t_f, \bp_{b}; \bp) \label{eq:beta0-def}\\
 \bpsi_f & = \blam(t_f) - \partial_{\bx_f}\overline{E}(\bnu; \bx_0, t_0, \bp_a; \bx_f, t_f, \bp_{b}; \bp) \label{eq:betaf-def}\\
\bpsi_0^T \bg_0(\bx_0, s_a) &=
 \partial_{s_a}\overline{E}(\bnu; \bx_0, t_0, \bp_a; \bx_f, t_f, \bp_{b}; \bp)
 + \nu_a^C \partial_{s_a}\phi_0(s_a, t_0) + \eta_{s_a}+\sigma_{s_a} \label{eq:ross-ext-lam0}\\
\bpsi_f^T \bg_f(\bx_f, s_b) &=
 \partial_{s_b}\overline{E}(\bnu; \bx_0, t_0, \bp_a; \bx_f, t_f, \bp_{b}; \bp)
 + \nu_b^C \partial_{s_b}\phi_f(s_b, t_f) + \eta_{s_b} + \sigma_{s_b} \label{eq:ross-ext-lamf}\\
\mathcal{H}[@t_0] &= \partial_{t_0}\overline{E}(\bnu; \bx_0, t_0, \bp_a; \bx_f, t_f, \bp_{b}; \bp) + \nu_{t_0} + \nu_a^C \partial_{t_0}\phi_0(s_a, t_0) \label{eq:ross-ext-hvc-0}\\
- \mathcal{H}[@t_f] &=  \partial_{t_f}\overline{E}(\bnu; \bx_0, t_0, \bp_a; \bx_f, t_f, \bp_{b}; \bp) + \nu_{t_f} + \nu_b^C \partial_{t_f}\phi_f(s_b, t_f) \label{eq:ross-ext-hvc-f}\\
 \partial_{\bx^\sharp_a}\overline{E}(\bnu; \bx_0, t_0, \bp_a; \bx_f, t_f, \bp_{b}; \bp)  &=
\bpsi_0 - \etab_{x^\sharp_a} - \bsigma_{x^\sharp_a} \label{eq:ross-ext-param-a1}\\
 \partial_{s^\sharp_a}\overline{E}(\bnu; \bx_0, t_0, \bp_a; \bx_f, t_f, \bp_{b}; \bp) &=
- \left(\bpsi_0^T\bg_0(\bx^\sharp_a, s^\sharp_a) + \eta_{s^\sharp_a} + \sigma_{s^\sharp_a}\right) \label{eq:ross-ext-param-a2}\\
\partial_{\bx^\sharp_b}\overline{E}(\bnu; \bx_0, t_0, \bp_a; \bx_f, t_f, \bp_{b}; \bp)  &=
\bpsi_f - \etab_{x^\sharp_b} - \bsigma_{x^\sharp_b}\\
\partial_{s^\sharp_b}\overline{E}(\bnu; \bx_0, t_0, \bp_a; \bx_f, t_f, \bp_{b}; \bp) &=
-\left(
\bpsi_f^T\bg_f(\bx^\sharp_b, s^\sharp_b)  + \eta_{s^\sharp_b} \right)\\
-\bsigma & = \partial_{\bp} \overline{E}(\bnu; \bx_0, t_0, \bp_a; \bx_f, t_f, \bp_{b}; \bp)\label{eq:ross-ext-p}
\end{align}
\end{subequations}
\end{theorem}
The proof of Theorem~\ref{theorem:ross-ext} is fairly straightforward in terms of repeating all the steps delineated in Section~\ref{sec:proof}.
\begin{remark}
Because the constraints on $t_0$ and $t_f$ are already included in the definition of $S_0^{ode+}$ and $S_f^{ode+}$ respectively, it seems superfluous to include additional constraints on these clock times via $S^{alg}$ in \eqref{eq:endpt4B++}.  In fact, at first glance, it appears that including additional clock-time constraints in $S^{alg}$ might lead to the condition $(S_0^{ode+} \times S_f^{ode+})\cap S^{alg} = \emptyset$.  In principle, this is true, which is why we assume $(S_0^{ode+} \times S_f^{ode+})\cap S^{alg} \neq \emptyset$.  However, note that the clock-time constraints required in the definitions of $S_0^{ode+}$ and $S_f^{ode+}$ are independent of each other.  The set $S^{alg}$ allows $t_0$ and $t_f$ to be jointly constrained (via the endpoint function $\be$) thereby adding an additional layer of practicality.  In fact, in \cite{dixon-diss-2025} the flight time $t_f -t_0$ is constrained in one of the problem definitions.  This flight-time constraint cannot be imposed by separately constraining the clock times $t_0$ and $t_f$ in the case of the elliptic, restricted three-body problem and hence also the general $N$-body problem.
\end{remark}

\section{Conclusions}

Conventionally, boundary conditions are defined in terms of algebraic equations. The main problem with this convention is that it assumes integrals of motion, equilibria, or some other special reductions of the differential equations. Per the illustrative three-body problem, it is apparent that there was a fundamental gap in the literature in framing boundary conditions in forms other than algebraic inequalities or abstract sets. To the best of the author's knowledge, this is the first paper to explore and define boundary conditions in terms of differential equations. Once the unconventional idea of boundary differential equations is accepted, the main task is to formulate this mathematical problem. As shown in this paper, formulating this new problem requires one to distinguish and separate local time-like variables and the transfer clock times. Because these quantities are parameters in an optimal control problem, Pontryagin's original theorem rectified by subsequent researchers forms the foundation for the development of the new transversality conditions.  The concept of weak adjoint covectors associated with the boundary differential equations is indispensable to the production of the new transversality conditions.  Per Theorem~\ref{theorem:ross} and its corollaries developed in this paper,  the new transversality conditions can be easily computed because it is stated in terms of the vector field that defines the boundary differential equations. Consequently, they can be used in a computational setting to validate the extremality of a computed solution. It turns out that many practical problems require boundary conditions to be naturally stipulated in terms of both algebraic and differential equations rather than one or the other. This is what makes Theorem~\ref{theorem:ross-ext} the most practical of all the results developed in this paper.





\end{document}